\documentclass{amsart}

\usepackage{graphicx}
\usepackage{amsmath}
\usepackage{amsthm}
\usepackage{amssymb}
\usepackage{enumerate}
\usepackage{mathrsfs} 
\usepackage{hyperref}
\usepackage{xypic}
\usepackage[nooneline, figbotcap]{subfigure}

\usepackage{tikz}

\newtheorem{theorem}{Theorem}[section]
\newtheorem{lemma}[theorem]{Lemma}
\newtheorem{proposition}[theorem]{Proposition}
\newtheorem{claim}{Claim}
\newtheorem{corollary}[theorem]{Corollary}

\newtheorem*{thmA}{Theorem~\ref{thm:spectrum-homology-sphere_bound}}

\theoremstyle{definition}
\newtheorem{definition}[theorem]{Definition}

\theoremstyle{remark}
\newtheorem{remark}[theorem]{Remark}

\numberwithin{equation}{section}

\def\S{\Sigma}
\def\D{\mathcal{D}}
\def\Z{\mathbb{Z}}
\def\NN{\mathbb{N}}
\def\NI{\mathbb{N}\cup\{\infty\}}
\def\TT{\mathscr{T}}
\def\H{\mathscr{H}}

\subfiguretopcaptrue

\begin{document}

\title{The Triple Point Spectrum of closed orientable $3$-manifolds}

\author{\'Alvaro Lozano Rojo}
\address{Centro Universitario de la Defensa Zaragoza, Academia General Militar
Carretera de Huesca s/n. 50090 Zaragoza, Spain --- IUMA, Universidad de Zaragoza}
\email{alvarolozano@unizar.es}

\thanks{Partially supported by the European Social Fund and Diputaci\'on General de Arag\'on (Grant E15 Geometr{\'\i}a).}

\author{Rub\'en Vigara Benito}
\address{Centro Universitario de la Defensa Zaragoza, Academia General Militar
Carretera de Huesca s/n. 50090 Zaragoza, Spain --- IUMA, Universidad de Zaragoza}
\email{rvigara@unizar.es}




\begin{abstract}
  The triple point numbers and the triple point spectrum of a closed 
  $3$-manifold were defined in~\cite{tesis}.
  They are topological invariants that give a measure of the complexity of a 
  $3$-manifold
  using the number of triple points of minimal filling Dehn surfaces. 
  Basic 
  properties of these invariants are presented, and the triple point spectra of 
  $\mathbb{S}^2\times \mathbb{S}^1$ and $\mathbb{S}^3$ are computed.
\end{abstract}

\maketitle

\section{Introduction}\label{sec:intro}

Through the whole paper all 3-manifolds and surfaces are assumed to be 
orientable and closed, that is, compact connected and without boundary, $M$ 
denotes a 3-manifold and $S$ a surface. All objects are assumed to be in the 
smooth category: manifolds have a differentiable structure and all maps are 
assumed to be smooth. By simplicity a genus $g$ surface, $g=0,1,2,\ldots$, is 
called a \emph{$g$-torus}.

Let $M$ be a 3-manifold.
  
A subset $\S\subset M$ is a \emph{Dehn surface} in $M$~\cite{Papa} if there 
exists a surface $S$ and a general position immersion $f:S\rightarrow M$ such 
that $\Sigma=f\left( S\right)$. In this situation we say that $S$ is the 
\emph{domain} of $\S$ and that $f$ \emph{parametrizes} $\Sigma$. If $S$ is a 
$2$-sphere, a torus, or a $g$-torus then $\Sigma$ is a \emph{Dehn sphere}, a 
\emph{Dehn torus}, or a \emph{Dehn $g$-torus} respectively. For a Dehn surface 
$\Sigma\subset M$, its singularities are divided into \emph{double points}, 
where two sheets of $\Sigma$ intersect transversely, and \emph{triple points}, 
where three sheets of $\S$ intersect transversely, and they are arranged along 
\emph{double curves} (see Section~\ref{sec:Dehn-surfaces-Johansson-diagrams} 
below for definitions and pictures).
  
  A Dehn surface $\S\subset M$ \emph{fills} $M$~\cite{Montesinos} if it defines a cell-decomposition of $M$ in which the 0-skeleton is the set of triple points of $\Sigma$, the 1-skeleton is the set of double and triple points of $\Sigma$, and the 2-skeleton is $\Sigma$ itself. Filling Dehn spheres are introduced (following ideas of W. Haken~\cite{Haken1}) in~\cite{Montesinos}, where it is proved that every closed, orientable 3-manifold has a filling Dehn sphere. In ~\cite{Montesinos} filling Dehn spheres and their Johansson diagrams (cf. Section~\ref{sec:Dehn-surfaces-Johansson-diagrams}) are proposed as a suitable way for representing all closed orientable 3-manifolds. A weaker version of filling Dehn surfaces are the \emph{quasi-filling Dehn surfaces} defined in~\cite{Amendola02}. A quasi-filling Dehn surface in $M$ is a Dehn surface whose complementary set in $M$ is a disjoint union of open $3$-balls. In~\cite{FennRourke} it is proved that every $3$-manifold has a quasi-filling Dehn sphere. As it is 
pointed out in~\cite{Amendola02} (see also~\cite{Funar}), filling Dehn surfaces are closely related with \emph{cubulations} of 3-manifolds: cell decompositions where the building blocks are cubes. The dual decomposition of that defined by a filling Dehn surface is a cubulation and vice versa, the dual cell decomposition of a cubulation is a filling Dehn surface (perhaps with nonconnected or nonorientable domain).
   \begin{figure}
    \centering
    \includegraphics[width=0.4\textwidth]{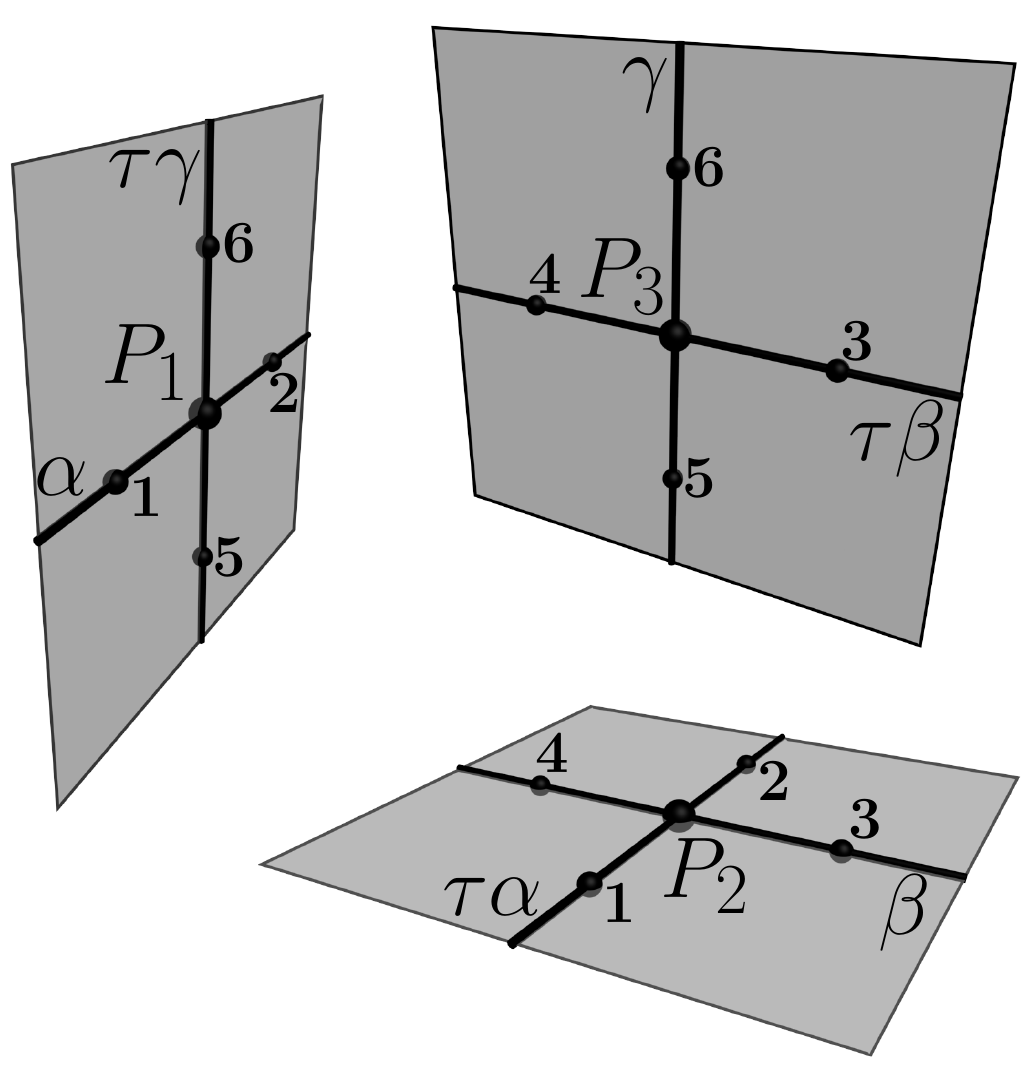}
    \includegraphics[width=0.1\textwidth]{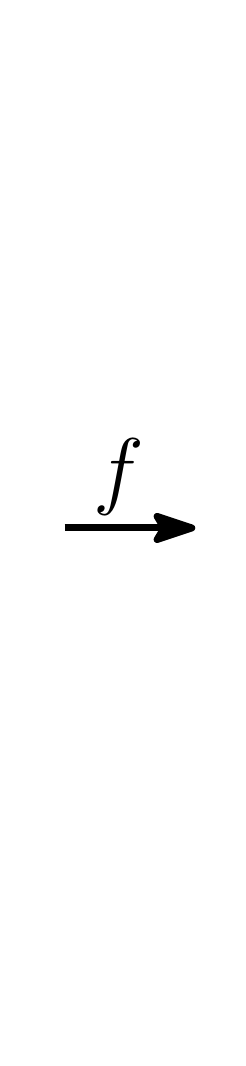}
    \includegraphics[width=0.3\textwidth]{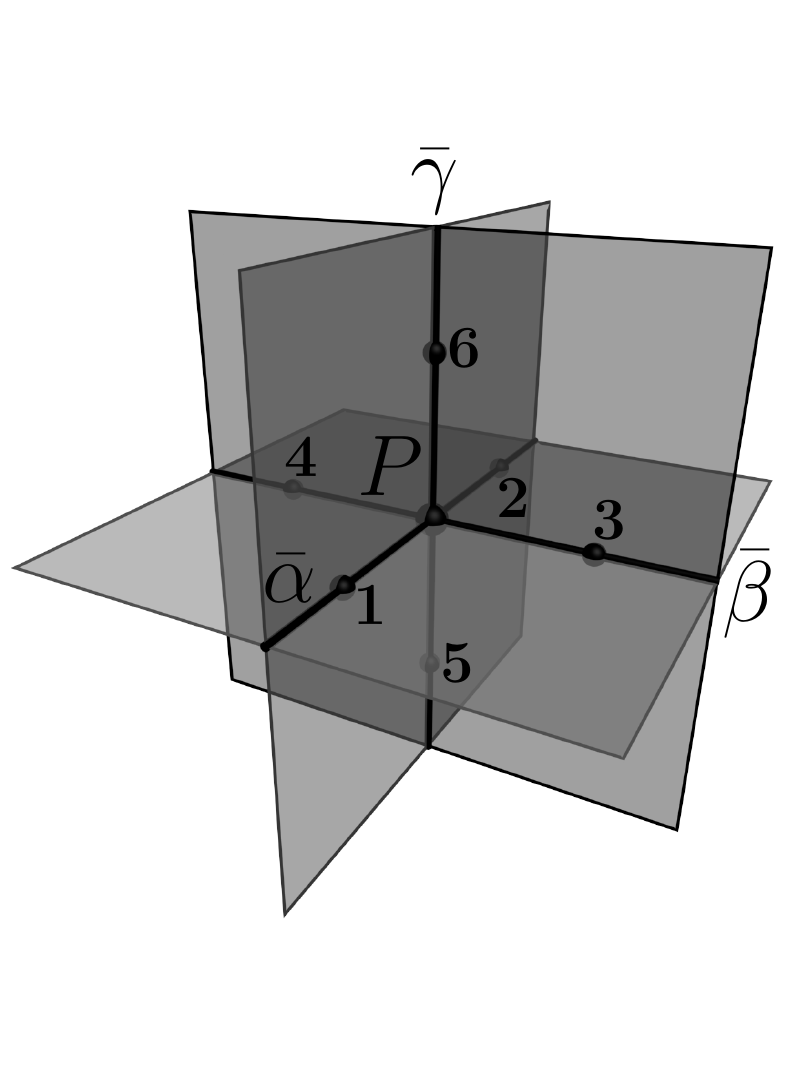}
    \caption{A triple point of $\S$ and its triplet
    in $S$}\label{fig:triple-point-00}
  \end{figure}
The number of triple points of filling Dehn surfaces measures the complexity of 
$3$-manifolds. A filling Dehn surface $\S\subset M$ with domain $S$ is 
\emph{minimal} if there is no filling Dehn surface in $M$ with domain $S$ with 
less triple points than $\S$. 

In~\cite{RHomotopies} it is defined the \emph{genus $g$ triple point number} 
$t_g(M)$ of $M$ as the number of triple points of a minimal filling $g$-torus 
of $M$, and the \emph{triple point spectrum} of $M$ as the sequence of numbers 
\[
  \mathscr{T}(M)=(t_0(M),t_1(M),t_2(M),\ldots)\,.
\]
The genus $0$ triple point number of $M$ is the \emph{Montesinos complexity} 
$M$~\cite{racsam}. These invariants are related to the 
\emph{surface-complexity} of $M$, defined in~\cite{Amendola02} as the number of 
triple points of a quasi-filling Dehn surface of $M$ (perhaps with nonconnected 
or nonorientable domain).
  
We want to clarify how much information about $M$ encloses $\TT(M)$. Since the 
original work~\cite{RHomotopies}, nothing is known about the triple point 
numbers and the triple point spectrum apart from their definitions. In this 
paper we introduce some properties of $\TT(M)$ 
(Section~\ref{sec:spectrum-surgerysurfaces}), and we present the two unique 
known triple point spectra
  \begin{align}
  \TT(\mathbb{S}^2\times \mathbb{S}^1)&=(2,1,3,5,\ldots)\,,\label{eq:triple-point-spectrum-S2-S1}\\
  \TT(\mathbb{S}^3)&=(2,4,6,8,\ldots)\,.\label{eq:triple-point-spectrum-S3}
  \end{align}
  
  The proof of \eqref{eq:triple-point-spectrum-S2-S1} relies on two previously known examples of minimal filling Dehn surfaces of $\mathbb{S}^2\times \mathbb{S}^1$, together with the inequality (Lemma~\ref{lem:triple-point-inequality})
  \begin{equation}
      t_{g+1}(M)\leq t_{g}(M)+2\,,\; g=0,1,2,\ldots ,\label{eq:consecutive-triple-point-numbers-2}
  \end{equation}
  that arises from a simple surgery operation (the \emph{handle piping}, see Figure~\ref{Fig:handle}). This will be detailed in Section~\ref{sec:spectrum-surgerysurfaces}.
  
  The characterization of the triple point spectrum of $\mathbb{S}^3$ needs 
  more work.
  The identity \eqref{eq:triple-point-spectrum-S3} is proposed in~\cite{racsam} 
  as an open question. 
  Using that $t_0(\mathbb{S}^3)=2$~\cite{racsam}, the handle piping provides $g$ filling
  Dehn $g$-tori in $\mathbb{S}^3$ with $2+2g$ triple points for all $g=1,2,\ldots$.
  The harder part is to prove that all these filling Dehn surfaces in $\mathbb{S}^3$ are minimal.
  To do so, in Section~\ref{sec:fundamentalgroup} we introduce a presentation of the 
  fundamental group a Dehn $g$-torus. This is a generalization of a presentation of 
  the fundamental group of Dehn spheres due to W. Haken and detailed in~\cite{tesis}. 
  Using it, in Section~\ref{sec:checkers-homology-spheres} we prove:

\begin{thmA}
  If $M$ is a $\Z/2$-homology $3$-sphere
  \begin{equation*}
    t_g(M) \geq 2 + 2g.
  \end{equation*}
\end{thmA}

  This theorem completes the proof of \eqref{eq:triple-point-spectrum-S3}.
  \begin{figure}
    \centering
    \includegraphics[width=0.6\textwidth]
    {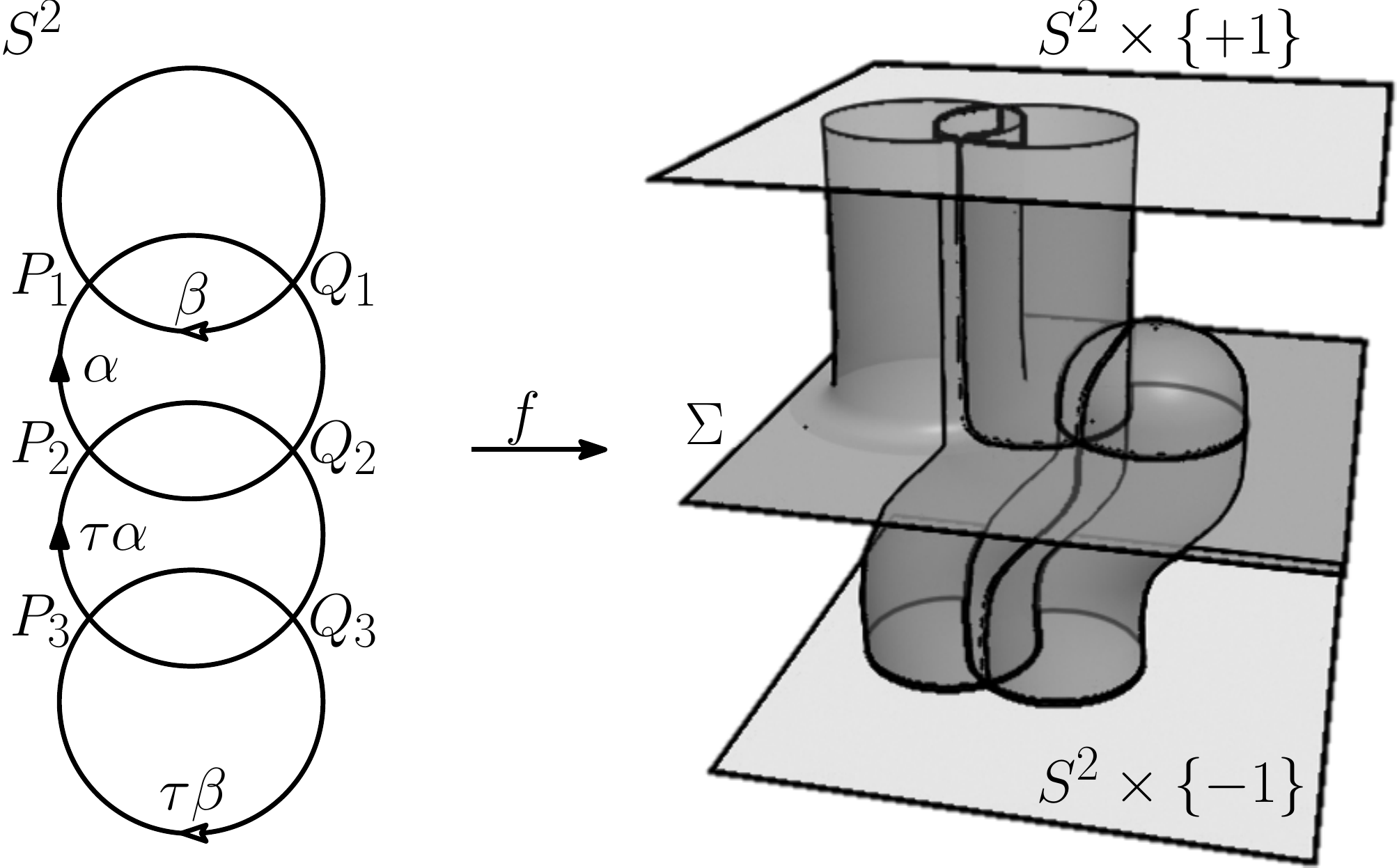}
    \caption{A filling Dehn sphere in $\mathbb{S}^2\times \mathbb{S}^1$}
    \label{Fig:audi}
  \end{figure}

\section{Dehn surfaces and their Johansson's diagrams}\label{sec:Dehn-surfaces-Johansson-diagrams}

  Let $\S$ be a Dehn surface in $M$ and consider a parametrization $f:S\to M$ 
  of $\S$.
  
  The \emph{singularities} of $\S$ are the points $x\in\S$ such that 
  $\#f^{-1}(x)>1$. The \emph{singularity set} $S(\S)$ of $\S$ is the set of 
  singularities of $\S$. As $f$ is in general position, the singularities of 
  $\S$ can be 
  divided into double points ($\#f^{-1}(x)=2$),  and triple points 
  ($\#f^{-1}(x)=3$). Following~\cite{Shima}, we denote by $T(\Sigma)$ the set 
  of triple points of $\Sigma$. The preimage under $f$ in $S$ of the 
  singularity set of $\Sigma$, together with the information about how its 
  points become identified by $f$ in $\S$ is the \emph{Johansson diagram} $\D$ 
  of $\S$ (see~\cite{Johansson1,Montesinos}).  We say that two points of $S$ 
  are \emph{related} if they project onto the same point of $\S$.
    
  In the following a \emph{curve} in $S$, $\S$ or $M$ is the image of an immersion 
  (a \emph{parametrization} of the curve) from $\mathbb{S}^1$ or $\mathbb{R}$ 
  into $S$, $\S$ or $M$, respectively. A \emph{path} in $S$, $\S$ or $M$  is a map $\eta$
  from the interval $[0,1]$ into $S$, $\S$ or $M$, respectively. We say that $\eta$ \emph{joins} 
  $\eta (0)$ \emph{with} $\eta (1)$, or that $\eta$ \emph{starts} at 
  $\eta (0)$ and \emph{ends} at $\eta (1)$, and if $\eta(0)=\eta(1)$ we say that $\eta$
  is a \emph{loop}.
  A double curve of $\S$ is a curve in $M$ contained in $S(\S)$.
  
  Because $S$ is closed, double curves are closed and there is a finite number of them. The number of triple points is also finite. Because $S$ and $M$ are orientable, the preimage under $f$ of a double curve of $\S$ is the union of two different closed curves in $S$, and we will say that these two curves are \emph{sister curves} of $\D$. Thus, the Johansson diagram of $\S$ is composed by an even number of different closed curves in $S$, and we will  identify $\D$ with  the  set of different  curves that compose it. For any curve $\alpha\in\D$ we denote by $\tau \alpha$ the sister curve of $\alpha$ in $\D$. This defines a free involution $\tau:\D\rightarrow\D$, the \emph{sistering} of $\D$, that sends each curve of $\D$ into its sister curve of $\D$.
  
  The curves of $\D$ intersect with others or with themselves transversely at the \emph{crossings} of $\D$. The crossings of $\D$ are the preimage under $f$ of the triple points of $\S$. If $P$ is a triple point of $\S$, the three crossings of $\D$ in $f^{-1}(D)$ compose \emph{the triplet of} $P$  (see Figure~\ref{fig:triple-point-00}). In Section~\ref{sec:fundamentalgroup} we will consider paths contained in curves of $\D$ or in double curves. In this special case, we will consider that paths are also immersed, and therefore they continue ``straight ahead'' when they arrive to a crossing.
  
    If $\Sigma$ is a Dehn surface in $M$, a connected component of $M-\Sigma$ is a \emph{region} of $\Sigma$, a connected component of $\Sigma-S(\Sigma)$ is a \emph{face} of $\Sigma$, and a connected component of $S(\S)-T(\S)$ is an \emph{edge} of $\S$. The Dehn surface $\S$ fills $M$ if and only if all its edges, faces and regions are open $1$, $2$ or $3$-dimensional disks respectively.
  
  In Figure~\ref{Fig:audi} we have depicted (left) one of the simplest Johansson diagrams of filling Dehn spheres. This is the diagram of a Dehn sphere $\S$ (right) that fills $\mathbb{S}^2\times \mathbb{S}^1$. In this figure we consider $\mathbb{S}^2\times \mathbb{S}^1$ as $\mathbb{S}^2\times [-1,1]$ where the top and bottom covers $\mathbb{S}^2\times \{+1\}$ and $\mathbb{S}^2\times \{-1\}$ are identified by the vertical projection. This example appears completely detailed in~\cite[Example 7.5.1]{tesis}.
     
  If we are given an \emph{abstract diagram}, i.e., an even collection of curves 
  in $S$ coherently identified in pairs, it is possible to know if this 
  abstract diagram is \emph{realizable}: if it is actually the Johansson diagram 
  of a Dehn surface in a $3$-manifold (see~\cite{Johansson1,Johansson2,tesis}). 
  It is also possible to know if the abstract diagram is \emph{filling}: if it is 
  the Johansson diagram of a filling Dehn surface of a 3-manifold (see 
 ~\cite{tesis}). If $\S$ fills $M$, it is possible to build $M$ out of the Johansson diagram of $\S$. Thus, filling Johansson diagrams represent all closed, orientable $3$-manifolds.
  It must be noted that when a diagram $(\D,\tau)$ in $S$ is not realizable, the quotient space of $S$ under the equivalence relation defined by the diagram is something very close to a Dehn surface: it is a 2-dimensional complex with simple, double and triple points, but it cannot be embedded in any $3$-manifold. We reserve the name \emph{pseudo Dehn surface} for these objects. Many constructions about Dehn surfaces, as the presentation of their fundamental group given in Section~\ref{sec:fundamentalgroup}, for example, are also valid for pseudo Dehn surfaces.
      \begin{figure}
    \centering
    \includegraphics[width=0.8\textwidth]
    {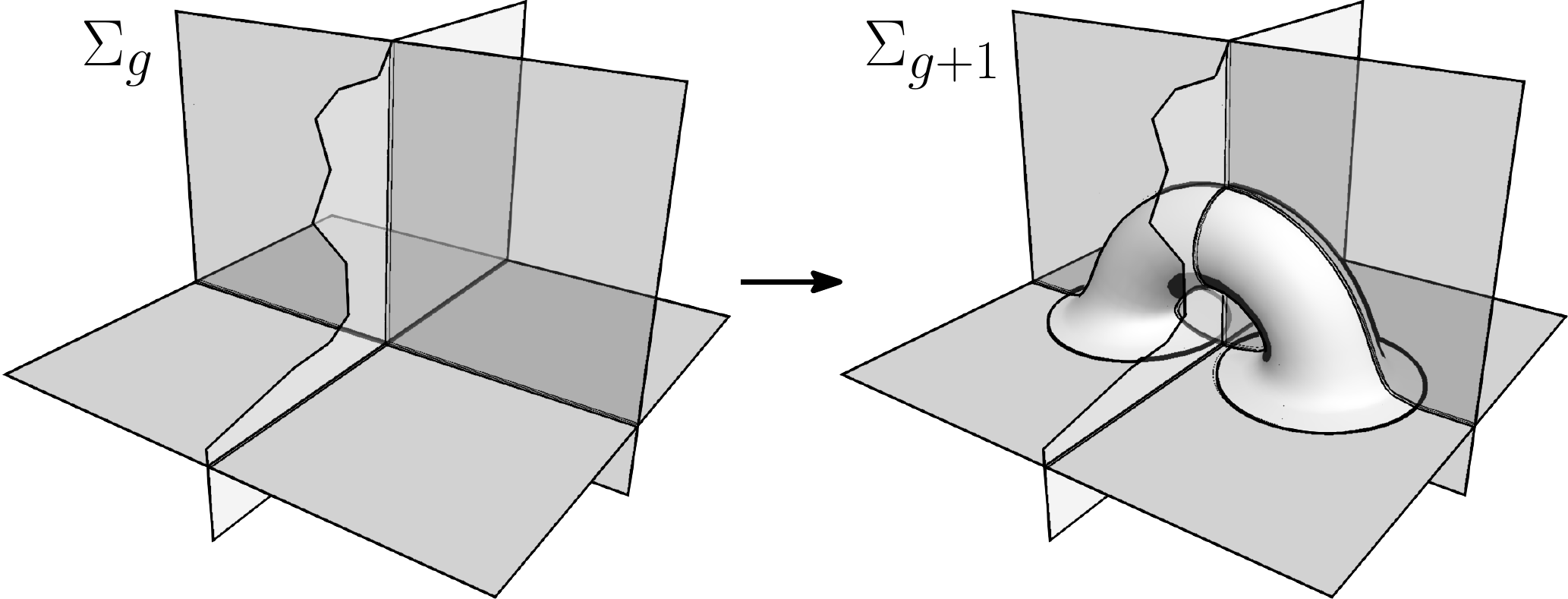}
    \caption{Handle piping}
    \label{Fig:handle}
  \end{figure}
  
    \section{The triple point spectrum}
  \label{sec:spectrum-surgerysurfaces}
  
  In order to understand the structure of the sequence $\TT(M)$, we need to relate the different triple point numbers of $M$. For a shorter notation we will omit the word ``filling'' in the expression ``minimal filling Dehn surface''. If we are given a filling Dehn $g$-torus $\S_g$ of $M$ with $q$ triple points, we can always obtain a filling Dehn $(g+1)$-torus $\S_{g+1}$ of $M$ with $q+2$ triple points by applying to $\S_g$ the \emph{handle piping} modification depicted in Figure~\ref{Fig:handle} in a small neighbourhood of a triple point. If the original $\S_g$ is minimal, we have:  
  \begin{lemma}\label{lem:triple-point-inequality}
  For any $g=0,1,2,\ldots$
   \begin{equation}
   \pushQED{\qed}%
      t_{g+1}(M)\leq t_{g}(M)+2\,,\label{eq:consecutive-triple-point-numbers-3}\qedhere
    \popQED%
   \end{equation}
  \end{lemma}
  This lemma suggests the following definitions.
  \begin{definition}\label{def:exceptional}
   A minimal Dehn $g$-torus of $M$ is \emph{exceptional} if it has less than $t_{g-1}(M)+2$ triple points.
  \end{definition}
  \begin{definition}\label{def:height}
   The \emph{height} $\H(M)$ of $M$ is the highest genus among all exceptional Dehn $g$-tori of $M$.
   If $M$ has no exceptional Dehn $g$-torus, $\H(M)=0$.
  \end{definition}
  It is clear that if $\H(M)$ is finite, the equality in \eqref{eq:consecutive-triple-point-numbers-3} holds for all $g\geq \H(M)$.
     
    An important result for filling Dehn surfaces which is of great interest here is the following.
    \begin{proposition}
	\label{prop:triplets-regions}
	A Dehn $g$-torus in $M$ with $q$ triple points and $r$ regions fills $M$ if and only if
	\begin{equation}
	  r=q+(2-2g)\,.\label{eq:triplets-regions}
	\end{equation}
    \end{proposition}
    This proposition follows from~\cite[Theorem 3.7.1]{tesis} (see also~\cite[Lemma 43]{RHomotopies}). Its proof relies on Euler's characteristic techniques and it has strong consequences in this context:
    \begin{remark}\label{rmk:at-least-1+2g-triple-points}
      A filling Dehn $g$-torus has at least $2g-1$ triple points.
    \end{remark}
    \begin{remark}\label{rmk:handle-regions-unchanged}
      After a handle piping on a filling Dehn $g$-torus the number of regions remains unchanged. More generally, if $\S_g$ is a filling Dehn $g$-torus of $M$ with $q$ triple points and $\S_{g+1}$ is a filling Dehn $(g+1)$-torus of $M$ with $q+2$ triple points, the number of regions of $\S_g$ and $\S_{g+1}$ is the same.
    \end{remark}
    \begin{remark}\label{rmk:exceptional-reduce-regions}
      Exceptional Dehn $g$-tori reduce the number of regions, that is, an exceptional Dehn $g$-torus 
      in $M$ has less regions than any minimal Dehn $g'$-torus in $M$ with $g'<g$.
    \end{remark}
  \begin{figure}
    \centering
    \subfigure[]{%
      \includegraphics[width=0.49\textwidth]{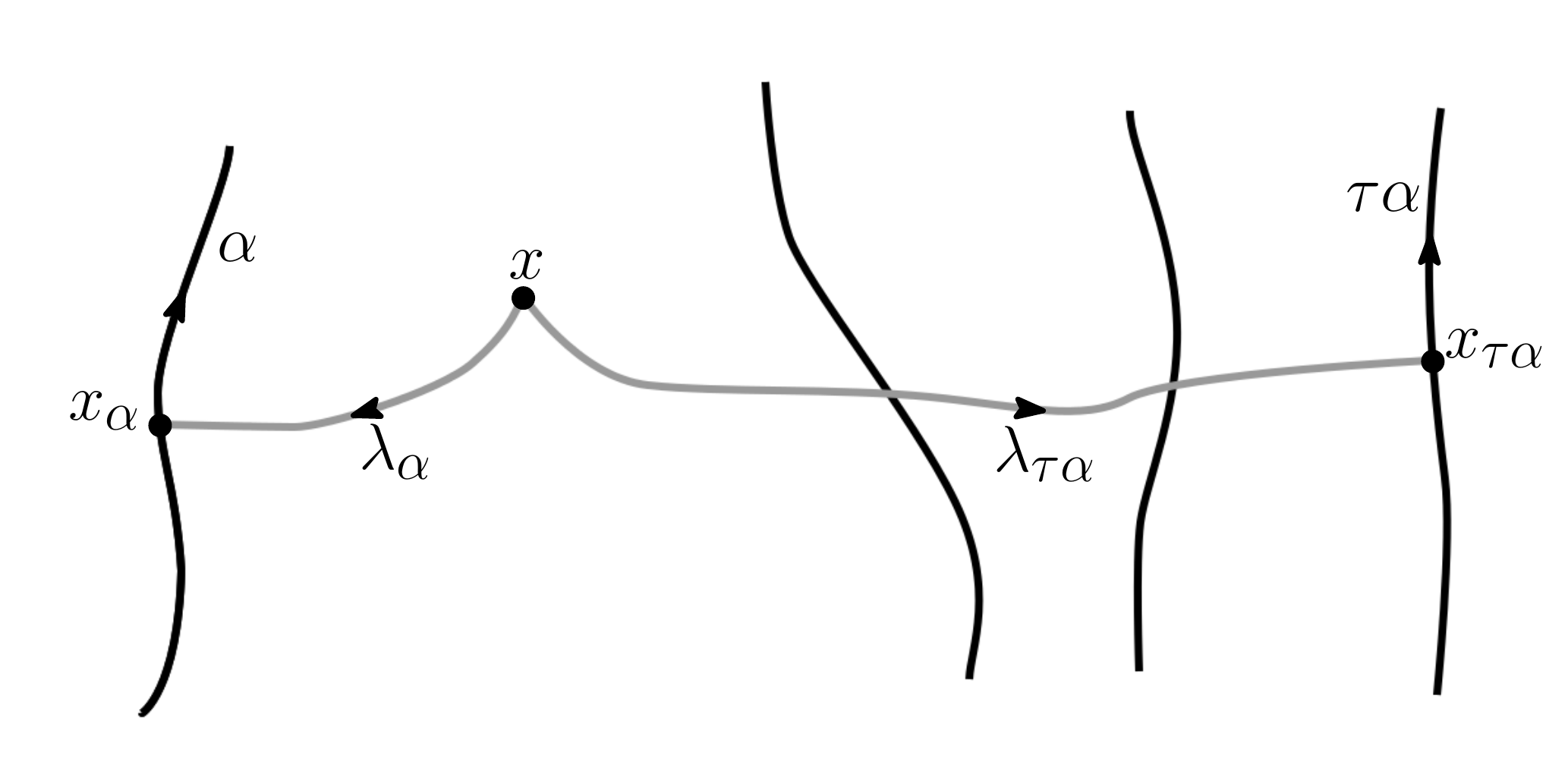}%
      \label{Fig:dual-curve-flat}}
    \subfigure[]{%
      \includegraphics[width=0.49\textwidth]{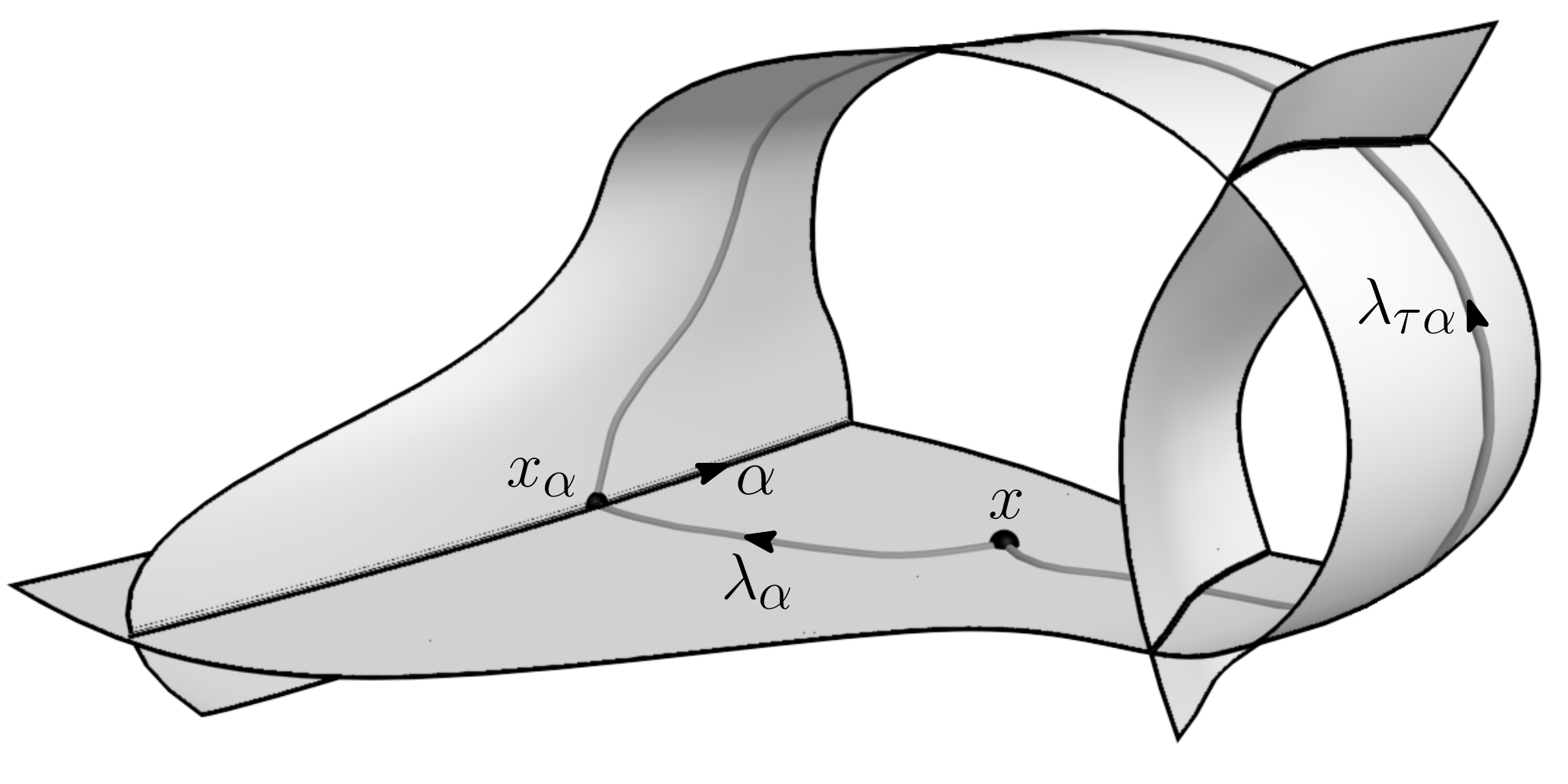}%
      \label{Fig:dual-curve}%
    }
    \caption{Dual curves to the curves of $\D$}
    \label{Fig:dual-curve-flat-dual-curve}
  \end{figure}
  
    As the number of regions is bounded from below, by Remark~\ref{rmk:exceptional-reduce-regions} there cannot be exceptional Dehn $g$-tori in $M$ with arbitrarily high genus. Therefore,
    \begin{theorem}\label{thm:height-finite}
    The height of $M$ is finite.\qed
    \end{theorem}
    
    In~\cite[Example 7.5.2]{tesis} it is shown a filling Dehn $1$-torus of $\mathbb{S}^2\times \mathbb{S}^1$ with just $1$ triple point. It is clearly a minimal Dehn $1$-torus, and by Proposition~\ref{prop:triplets-regions} it has only one region. By Remark~\ref{rmk:exceptional-reduce-regions}, there cannot be an exceptional Dehn $g$-torus in $\mathbb{S}^2\times \mathbb{S}^1$ of genus $g>1$. On the other hand, a Dehn sphere in any $3$-manifold has an even number of triple points~\cite{Haken1}, and therefore the filling Dehn sphere of Figure~\ref{Fig:audi} is minimal. It turns out that $\H(\mathbb{S}^2\times \mathbb{S}^1)=1$ and
    \begin{theorem}
     $\TT(\mathbb{S}^2\times \mathbb{S}^1)=(2,1,3,5,\ldots )\,.$\qed
    \end{theorem}

  \section{The fundamental group of a Dehn $g$-torus}
  \label{sec:fundamentalgroup}

  In this section the manifold $M$ containing $\S$ is no longer needed. Although we still talk about ``Dehn surfaces'', the construction only makes use of the Johansson diagram $\D$, so it is valid for pseudo Dehn surfaces.
  
  Let $f:S\to M$  be a Dehn surface in $M$ and denote by $\D$ its Johansson diagram.

  Fix a simple point $x\notin\D$ as the base point of the fundamental group
  $\pi_1(\Sigma)$ of $\Sigma$. We identify $x$ with its preimage 
  under $f$.

  A path in $\Sigma$ is \emph{surfacewise} if it is a path on $S$  mapped to
  $\Sigma$ through $f$. We denote by $\pi_S = f_*\pi_1(S)$ the subgroup of $\pi_1(\Sigma)$
  generated by the sufacewise loops based at $x$. 
  In general we will use the same notation for a surfacewise path in $\Sigma$ 
  and its preimage path in $S$.
 
  Let $\alpha,\tau\alpha$ be two sister curves of $\D$, and take two paths 
  $\lambda_{\alpha},\lambda_{\tau\alpha}$ in $S$, starting from $x$ and
  arriving to related points on $\alpha,\tau\alpha$, respectively
  (see Figure~\ref{Fig:dual-curve-flat-dual-curve}).

  \begin{definition}\label{def:dual-paths}
    We say that the loop $\lambda_\alpha \lambda_{\tau\alpha}^{-1}$ in
    $\Sigma$ based at $x$ is \emph{dual} to $\alpha$.	
  \end{definition}

  \begin{proposition}\label{prop:surfacewise-and-duals-generate-pi1}
    Surfacewise loops based at $x$ and loops dual to the curves 
    of $\D$ generate $\pi_1(\Sigma)$.
  \end{proposition}
  \begin{figure}
    \centering
    \subfigure[]{\includegraphics[width=0.48\textwidth]
      {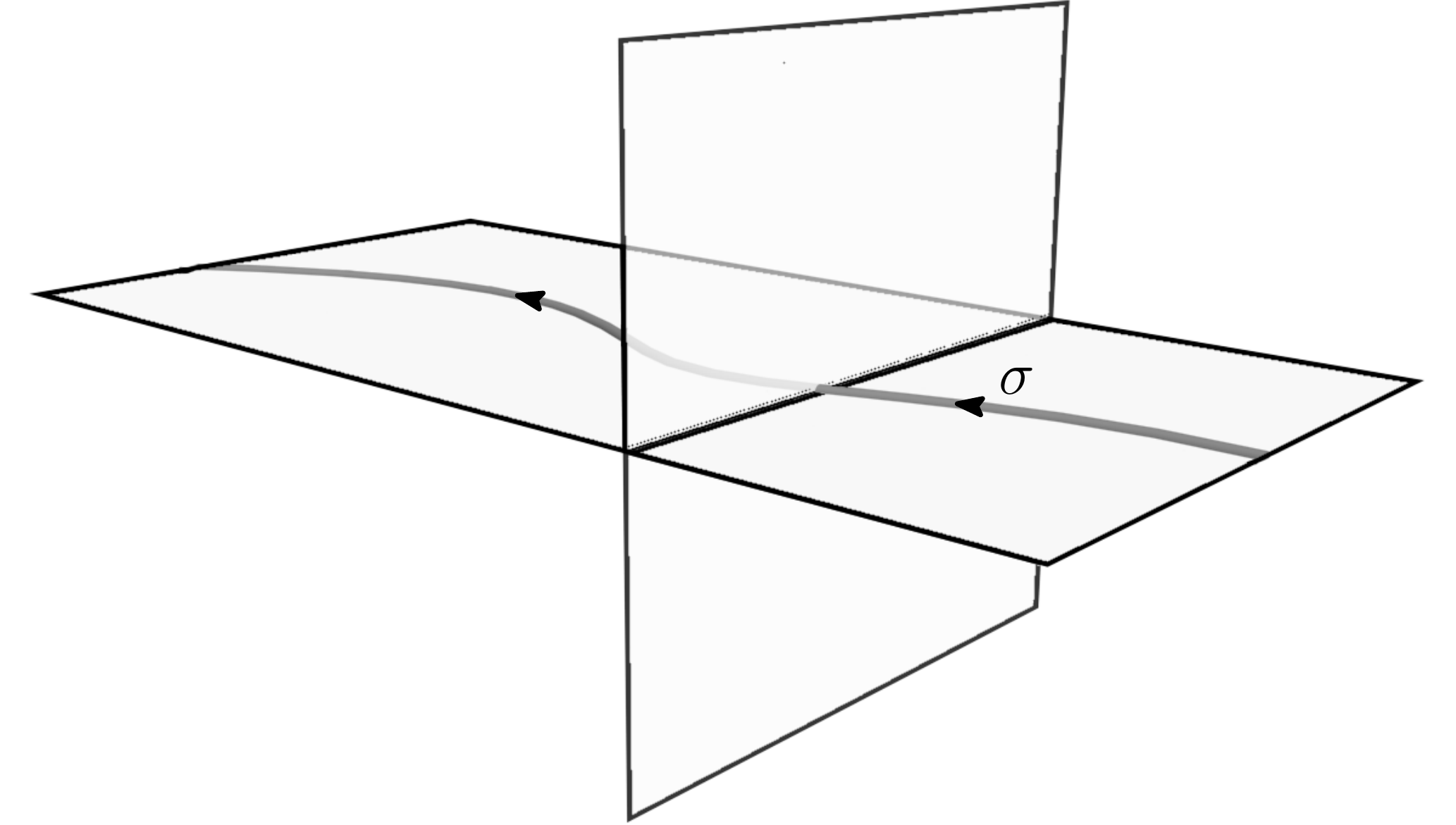}
      \label{Fig:crossing-type}}
    \subfigure[]{\includegraphics[width=0.48\textwidth]
      {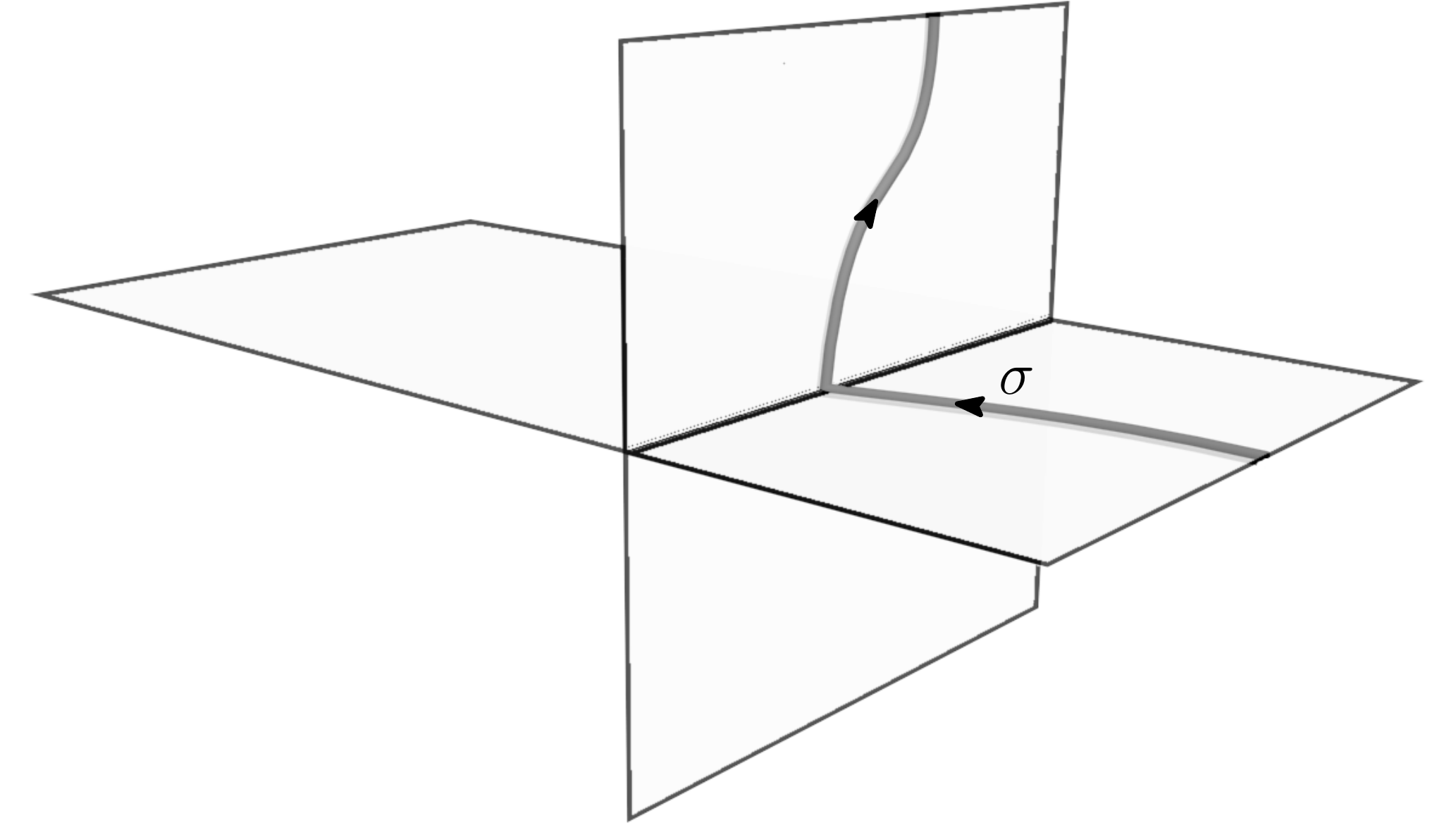}
      \label{Fig:corner-type}}
    \caption{Crossing-type and corner-type intersections with $S(\Sigma)$}
      \label{Fig:crossing-corner-type}
  \end{figure}
  
  \begin{proof}
    Consider a loop $\sigma$ in $\Sigma$ based at $x$. Up to homotopy, we can 
    assume that all the intersections of $\sigma$ with the singular set of $\Sigma$ are of 
    ``crossing type'' or of ``corner type'' as in Figure~\ref{Fig:crossing-type} and~\ref{Fig:corner-type} respectively.
    Therefore, $\sigma$ can be written as a product 
    $\sigma = b_1 b_2\cdots b_k$ of surfacewise paths
    (Figure~\ref{Fig:holed_plane}). If $k=1$ or $2$, either $\sigma$ is
    surfacewise or it is dual to a 
    curve of $\D$, and so there is nothing to prove.
    
    In other case, for each $i=2,\ldots,k-1$ we choose a midpoint $x_i$ of $b_i$
    and a path $c_i$ joining $x$ with $x_i$ (Figure~\ref{Fig:holed_plane}).
    We write $b_i=b_i^-b_i^+$, with
    $b_i^-$ ending and $b_i^+$ starting at $x_i$. The loop $\sigma$
    is homotopic to
    \[
      b_1 b_2^- c_2^{-1} c_2 b_2^+ b_3^-
        \cdots c_{k-2}b_{k-2}^+ b_{k-1}^- c_{k-1}^{-1} 
                c_{k-1} b_{k-1}^+ b_k\,, 
    \]
    and this expression is a product of loops dual to curves of $\D$.
  \end{proof}
    \begin{figure}
    \centering
    \includegraphics[width=0.7\textwidth]
    {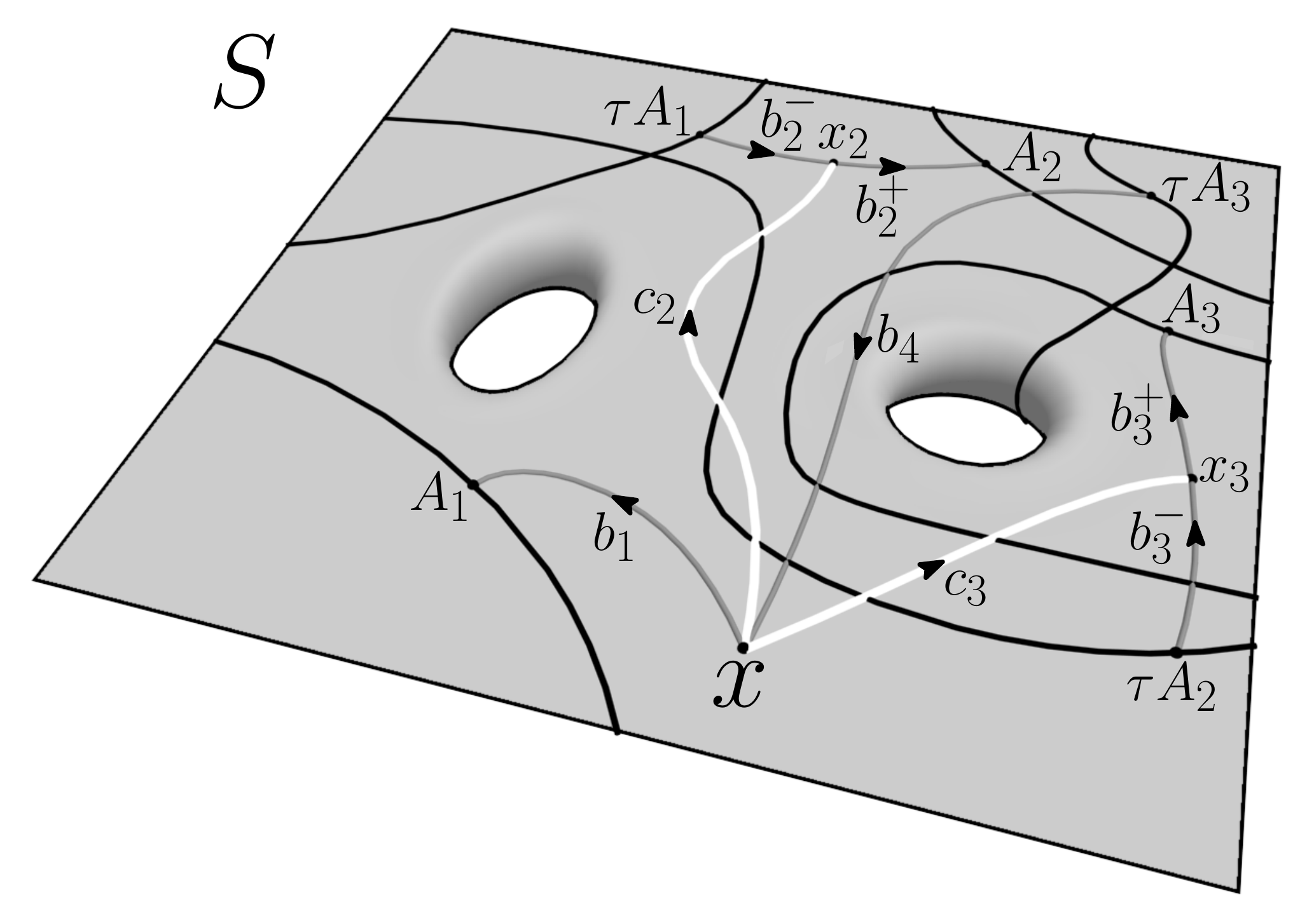}
    \caption{Loops as product of dual loops}
    \label{Fig:holed_plane}
  \end{figure}

  We know that $\pi_S$ is finitely generated because so is $\pi_1(S)$. The
  following lemma allows us to find a finite set of generators for 
  $\pi_1(\Sigma)$.

  \begin{lemma}
    \label{lem:fundamental-dual-paths-surface-conjugate}
    If $a,a'$ are loops on $\Sigma$ dual to $\alpha\in\D$, they are 
    \emph{surfacewise conjugate}, that is, there are two surfacewise loops 
    $s ,t$ based at $x$ such that $a=s  a' t$.
  \end{lemma}

    \begin{figure}
    \centering
    \includegraphics[width=0.7\textwidth]
{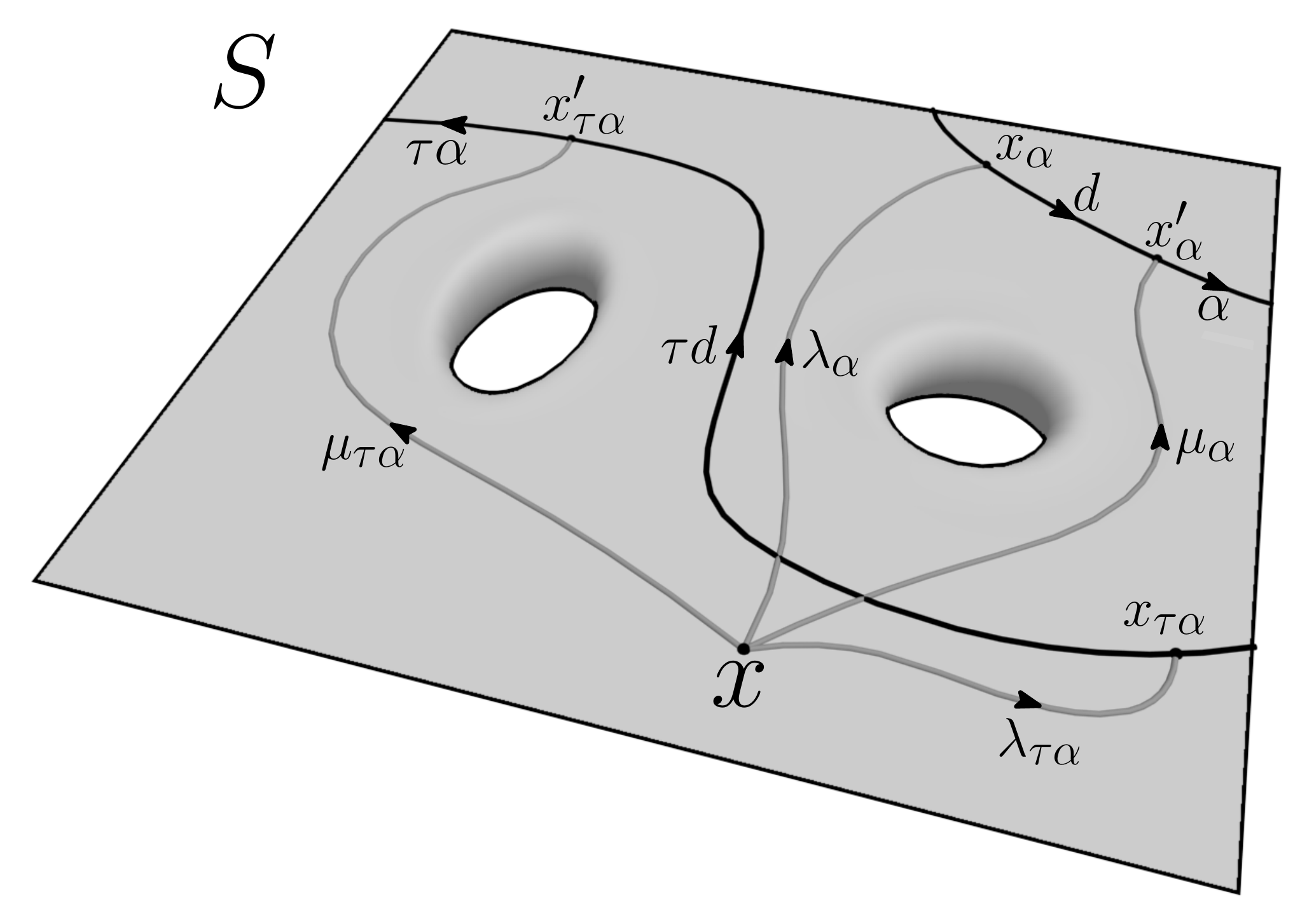}
\caption{Duals to $\alpha$ are
surfacewise conjugate}\label{Fig:holed_plane_duals-conjugate}
\end{figure}

  \begin{proof}
    Put $a=\lambda_\alpha \lambda_{\tau\alpha}^{-1}$, $a'=\mu_\alpha
    \mu_{\tau\alpha}^{-1}$, where 
    $\lambda_\alpha,\lambda_{\tau\alpha},\mu_\alpha,\mu_{\tau\alpha}$ are
    surfacewise paths as in Definition~\ref{def:dual-paths}. Let
    $x_\alpha,x_{\tau\alpha},x'_\alpha,x'_{\tau\alpha}$ be the endpoints of
    $\lambda_\alpha,\lambda_{\tau\alpha},\mu_\alpha,\mu_{\tau\alpha}$ in $\D$,
    respectively (Figure~\ref{Fig:holed_plane_duals-conjugate}). Let $d$ be one
    path in $S$ contained in $\alpha$ joining $x_\alpha$ with $x'_\alpha$, and let $\tau d$ be the
    sister path of $d$ inside $\tau\alpha$ joining $x_{\tau\alpha}$ with
    $x'_{\tau\alpha}$. In $\Sigma$ the paths $d$ and $\tau d$ coincide, and so we have:
    $$a=\lambda_\alpha \lambda_{\tau\alpha}^{-1}=\lambda_\alpha d (\tau d)^{-1}
    \lambda_{\tau\alpha}^{-1}=
    \lambda_\alpha d \mu_{\alpha}^{-1}\mu_\alpha\mu_{\tau\alpha}^{-1}
    \mu_{\tau\alpha} (\tau d)^{-1} \lambda_{\tau\alpha}^{-1}= s  a' t\,,$$
where $s =\lambda_\alpha d \mu_{\alpha}^{-1}$ and $t=\mu_{\tau\alpha}
(\tau d)^{-1} \lambda_{\tau\alpha}^{-1}$ are surfacewise loops in $\Sigma$
based at $x$.
\end{proof}

From now on, we will fix a set of generators $s_1,s_2,\ldots,s_{2g}$ of $\pi_S$,
and we will also fix a set of preferred dual loops to the curves of
$\D$ as follows. For each diagram curve
$\alpha\in\D$ we choose a \emph{basepoint} $x_\alpha$ and a joining
arc $\lambda_\alpha$
from $x$ to $x_\alpha$, such that the basepoints of sister curves are
related (see Figure~\ref{Fig:dual-curve-flat-dual-curve}). The preferred dual loop of $\alpha\in\D$ is
        \[
            a = \lambda_\alpha \lambda_{\tau\alpha}^{-1}\,.
        \]
If a curve of $\D$ is denoted with lowercase greek
letters
\[
  \alpha, \beta, \gamma, \ldots, 
              \alpha_1, \alpha_2, \ldots, \alpha_i, \ldots, \tau\alpha\,,
\]
its preferred dual loop will be denoted with
the corresponding lowercase roman letters
\[
  a, b, c,\ldots, a_1, a_2,\ldots, a_i,\ldots, \tau a\,.
\]
    \begin{figure}
    \centering
    \includegraphics[width=0.4\textwidth]
{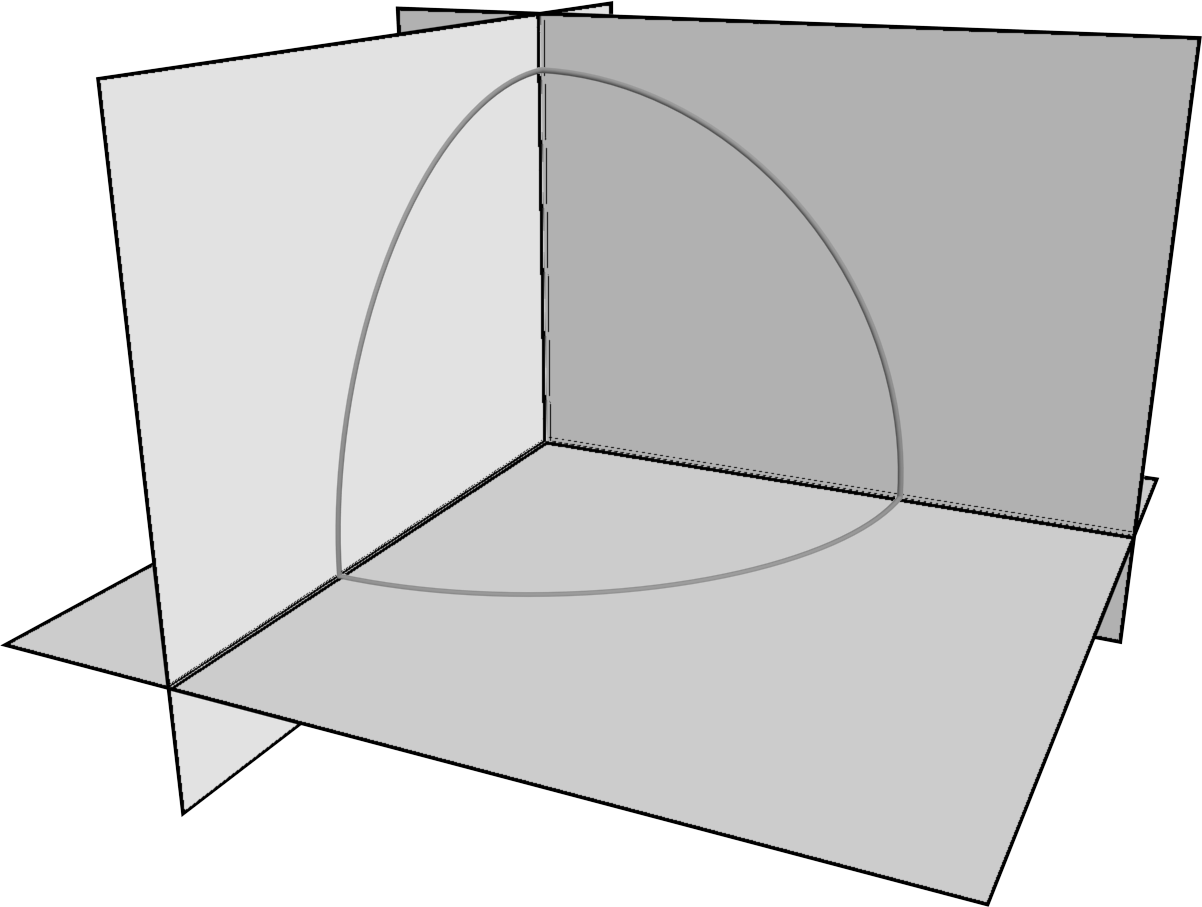}
\caption{Contractible loop near a triple point}\label{Fig:contractible-loop}
\end{figure}

By Proposition~\ref{prop:surfacewise-and-duals-generate-pi1} and Lemma
\ref{lem:fundamental-dual-paths-surface-conjugate}, we have:
\begin{theorem}\label{thm:finite-set-of-generators-pi1}
If $\alpha_1,\alpha_2,\ldots ,\alpha_{2k}$ are the curves of $\D$, then
$\pi_1(\Sigma)$ is generated by
\[
  \pushQED{\qed}%
  s_1,s_2,\ldots,s_{2g},a_1,a_2,\ldots,a_{2k}\,.\qedhere%
  \popQED%
\]
\end{theorem}

    \begin{figure}
    \centering
    \includegraphics[width=0.7\textwidth]
{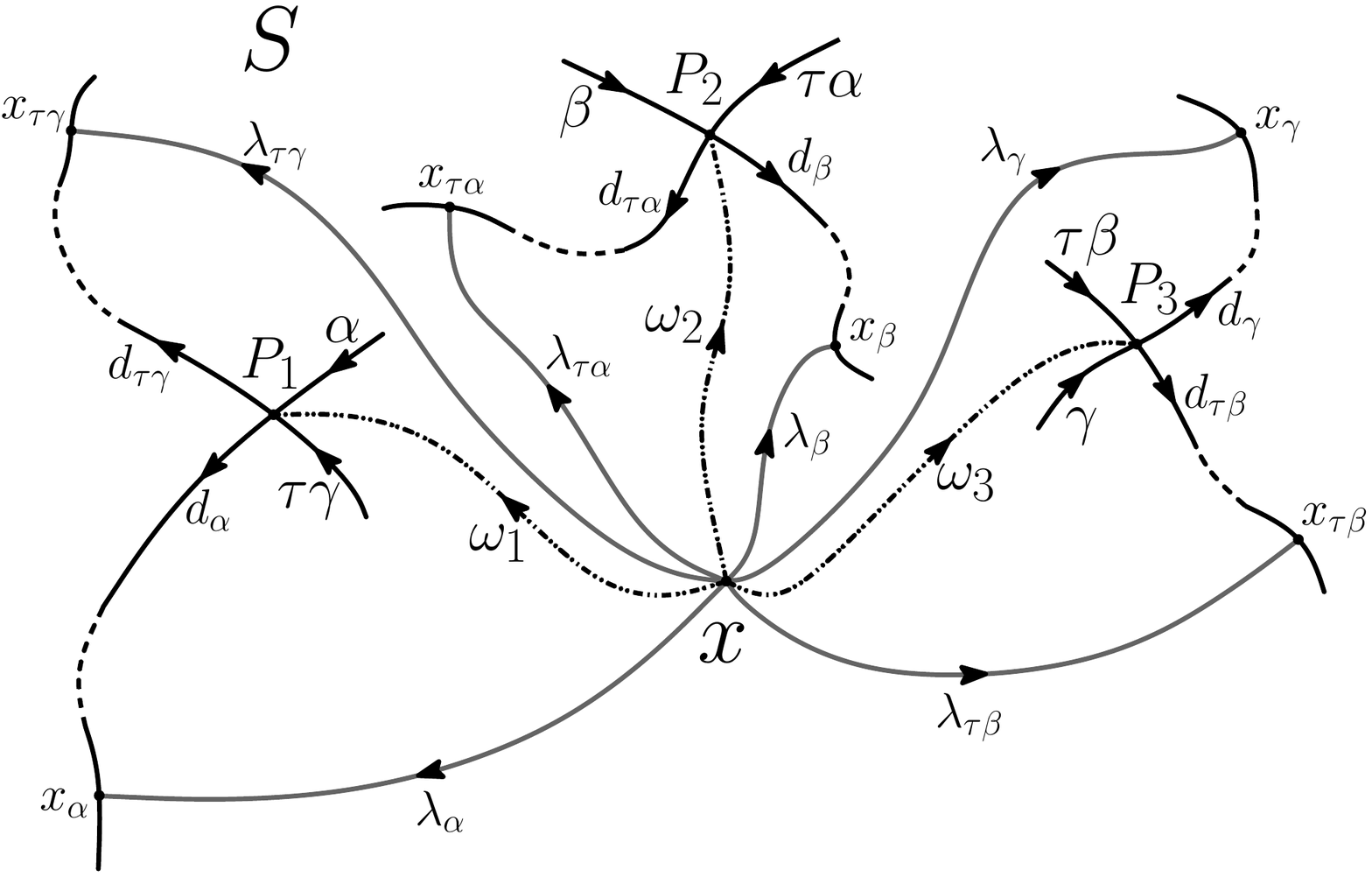}
\caption{Triple point relation (diagram view)}\label{Fig:triple-point-diagram}
\end{figure}

In order to obtain a presentation of $\pi_1(\Sigma)$, we need to establish a
set of relators associated to the generators of Theorem~\ref{thm:finite-set-of-generators-pi1}.
These relators will be a natural extension of Haken's relators for
Dehn spheres (see also~\cite{racsam,tesis}).
  \begin{enumerate}[\em(R1)]
    
    \item \emph{Dual loop relations}. By construction, if 
      $\alpha,\tau \alpha$ are two sister curves of $\D$, their dual loops 
      $a,\tau a$ verify $\tau a= a^{-1}$.

    \item \emph{Triple point relations}. The idea behind this relation is that any small circuit around
      a triple point as that of Figure~\ref{Fig:contractible-loop}
      is homotopically trivial.
      Assume that all the curves of $\D$ are oriented,
      with sister curves coherently oriented.
      Let $P$ be a triple point of $\Sigma$ and let $P_1,P_2,P_3$ be the 
      three crossings of $\D$ in the triplet of $P$. Label the curves of 
      $\D$ intersecting at these crossings as in 
      Figure~\ref{Fig:triple-point-diagram}.
      Consider three paths 
      $\omega_1,\omega_2,\omega_3$ joining $x$ with $P_1,P_2,P_3$ respectively. 
      At $P_1$, consider the path $d_\alpha$ contained in $\alpha$ that travels 
      from $P_1$ along the branch labelled $\alpha$ in the positive 
      sense until it reaches the basepoint $x_\alpha$ of $\alpha$. Consider 
      also 
      the similarly defined paths $d_\beta,d_\gamma$ contained in 
      $\beta,\gamma$ and joining $P_2,P_3$ with $x_\beta,x_\gamma$ 
      respectively, and the sister paths 
      $d_{\tau\alpha},d_{\tau\beta},d_{\tau\gamma}$ of 
      $d_{\alpha},d_{\beta},d_{\gamma}$ respectively, that join $P_2,P_3,P_1$ 
      with $x_{\tau\alpha},x_{\tau\beta},x_{\tau\gamma}$ along 
      $\tau\alpha,\tau\beta,\tau\gamma$, respectively. If we introduce the 
      loops
      \begin{align*}
        t_\alpha&=\omega_1 d_\alpha \lambda_\alpha^{-1}                       &
        t_\beta&=\omega_2 d_\beta \lambda_\beta^{-1}                          &
        t_\gamma&=\omega_3 d_\gamma \lambda_\gamma^{-1}                       \\
        t_{\tau\alpha}&=\omega_2 d_{\tau\alpha} \lambda_{\tau\alpha}^{-1}     &
        t_{\tau\beta}&=\omega_3 d_{\tau\beta} \lambda_{\tau\beta}^{-1}        &
        t_{\tau\gamma}&=\omega_1 d_{\tau\gamma} \lambda_{\tau\gamma}^{-1}
      \end{align*}
      we get the triple point relation around $P$:
      \begin{equation}
        \label{ec:3pto_relation}
        t_\alpha a t_{\tau\alpha}^{-1} \,
            t_\beta b t_{\tau\beta}^{-1} \,t_\gamma c t_{\tau\gamma}^{-1}=1\,.
      \end{equation}
      Note that the loops 
      $t_\alpha,t_{\tau\alpha}, t_\beta,t_{\tau\beta}, t_\gamma$ and 
      $t_{\tau\gamma}$ are surfacewise elements of $\Sigma$ and so they can be
      expressed as words in $s_1,s_2,\ldots,s_{2g}$.
      
    \begin{figure}
    \centering
      \includegraphics[width=0.7\textwidth]
      {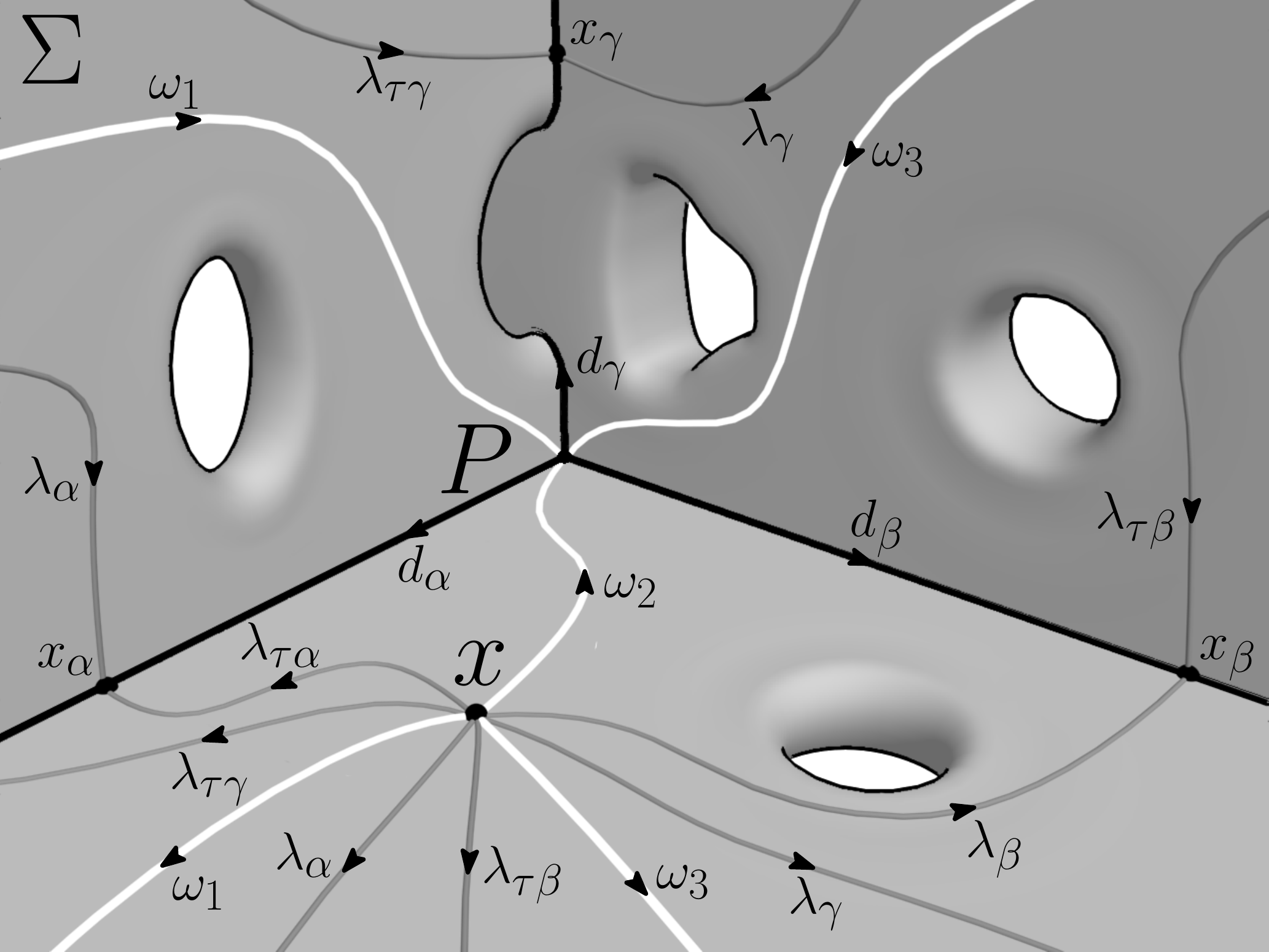}
      \caption{Triple point relation (Dehn surface
      view)}\label{Fig:triple-point-Sigma}
    \end{figure}

    \item \emph{Double curve relations}.
      In~\cite{racsam,RHomotopies} this relators did not appear because in $\mathbb{S}^2$ all loops 
      are homotopically trivial, while in a $g$-torus this is not true.
      Let $\alpha,\tau\alpha$ be a pair of sister curves 
      of $\D$. We orient both curves coherently starting from their basepoints 
      $x_\alpha,x_{\tau\alpha}$. The loops 
      $\lambda_\alpha \alpha \lambda_{\alpha}^{-1}$ and 
      $\lambda_{\tau\alpha} \tau\alpha \lambda_{\tau\alpha}^{-1}$ in $S$ are 
      related in $\Sigma$ because $\alpha$ and $\tau\alpha$ coincide in 
      $\Sigma$:
      \[
        \lambda_\alpha \alpha \lambda_{\alpha}^{-1}=
\lambda_\alpha\lambda_{\tau\alpha}^{-1} \lambda_{\tau\alpha} \tau\alpha
\lambda_{\tau\alpha}^{-1} \lambda_{\tau\alpha} \lambda_{\alpha}^{-1}\,.
      \]
      If we also denote by $\alpha$ and $\tau\alpha$ the surfacewise elements 
      of $\pi_1(\Sigma)$ represented by 
      $\lambda_\alpha \alpha \lambda_{\alpha}^{-1}$ and
      $\lambda_{\tau\alpha} \tau\alpha \lambda_{\tau\alpha}^{-1}$, 
      respectively, this relation can be written as
      \begin{equation}\label{eq:double-curve-relator}
        \alpha=a \,\tau\alpha \, a^{-1}\,. 
      \end{equation}

    \item \emph{Surface relations}. Those are the relations that the
      surfacewise generators $s_1,s_2,\ldots,s_{2g}$ verify when considered as 
      elements of $\pi_1(S)$. 
  \end{enumerate}
  Any relation among those specified above define an associated 
  \emph{dual loop relator}, \emph{triple point relator}, 
  \emph{double curve relator} or \emph{surface relator} in the natural way.

  \begin{theorem}
    \label{thm:fundamental-group-presentation}
    If
    $\D=\{\alpha_1,\alpha_2,\ldots,\alpha_{2k}\}$, the fundamental group
    $\pi_{1}(\Sigma)$ is isomorphic to
    \[
      \pi (\D)=\langle\, s_1,s_2,\ldots,s_{2g},a_1,a_2,\ldots,a_{2k} \mid
      \text{R1, R2, R3 and R4 relators}\,\rangle\,.
    \]
  \end{theorem}

  We need to introduce some notation and a previous result before being able
  to give the proof of this theorem. 
  If we consider the generators of $\pi (\D)$ as elements of $\pi_1 (\Sigma)$,
  we are defining implicitly a homomorphism
  $\varepsilon:\pi (\D)\to\pi_1(\Sigma)$ which is surjective by 
  Theorem~\ref{thm:finite-set-of-generators-pi1}. To prove that   
  $\varepsilon$ is also injective we will prove first that $\pi(\D)$ behaves 
  exactly as $\pi_1(\Sigma)$ when dealing with covering spaces:
  the covering spaces of $\Sigma$ are completely characterized by 
  representations of $\pi(\D)$ on a permutation group. For $m\in\NN$ we 
  denote by $\Omega_m$ the group of permutations of the finite set 
  $\{1,2,\ldots, m\}$, and we denote by $\Omega_\infty$ the group of 
  permutations of the countable set $\NN$.
  
  \begin{lemma}
    \label{lem:representations-factorize}
    Every representation $\rho:\pi (\D)\to \Omega_m$, with $m\in\NI$, factors 
    through $\varepsilon$, i.e., there exists a representation 
    $\hat\rho:\pi_1(\Sigma)\to\Omega_m$ such that the following diagram 
    commutes.
    \[
      \xymatrix{
        \pi(\D) \ar[d]_{\varepsilon} \ar[r]^{\rho} & \Omega_m \\
        \pi_1(\Sigma) \ar[ur]_{\hat\rho}
      }
    \]
  \end{lemma}

\begin{proof}
  Fix an $m\in\NI$, and consider a representation $\rho:\pi(\D)\to \Omega_m$.
  As $\pi(\D)$ verifies the surface relations we have a
  natural homomorphism
  \[
    \zeta:\pi_1(S)\to \pi(\D)\,.
  \]
  The map $\rho_S=\rho\circ \zeta$ is a representation of $pi_1(S)$, it is the 
  monodromy homomorphism of a $m$-sheeted covering map $h:\hat{S}\to S$.

  Let $x^1,x^2,\ldots$ be the different points of $\hat{S}$ on the fiber of $h$ 
  over $x$, labelled in such way that for any $s\in\pi_1(S)$ the lift of $s$ 
  starting in $x^i$ has its end in $x^j$, where $j=\rho_S(s)(i)$.

If we lift the diagram $\D$ to $\hat{S}$, we obtain another diagram $\hat{\D}$ in $\hat{S}$.
We will define a ``sistering'' $\hat{\tau}:\hat{\D}\to\hat{\D}$ between the
curves of $\hat{\D}$, with the expected compatibility with $\tau$. For any
curve $\alpha\in\D$, we call $\lambda_\alpha^i$ the lift of $\lambda_\alpha$
starting at $x^i$, and $x_\alpha^i$ the endpoint of $\lambda_\alpha^i$. The
sistering $\hat{\tau}$ is defined as follows: $\hat{\tau}$ sends the curve of
$\hat{\D}$ passing through $x_\alpha^i$ to the curve of $\hat{\D}$ passing
through $x_{\tau\alpha}^j$, where $j=\rho(a)(i)$. 
The points of both curves are related by $\hat{\tau}$ starting from $x_\alpha^i,x_{\tau\alpha}^j$
lifting the relation among the the points of $\alpha$ and $\tau\alpha$ near 
$x_\alpha$ and $x_{\tau\alpha}$ in the natural way. If we think
about the circle $\mathbb{S}^1$ as the real numbers modulo $2\pi$,
we can take parametrizations $\alpha,\tau\alpha:\mathbb{R}\to S$ of the curves 
$\alpha,\tau\alpha$ respectively such that $\alpha(0)=x_\alpha$, $\tau\alpha(0)=x_{\tau\alpha}$
and such that $\alpha(t)$ is related by $\tau$ with $\tau\alpha(t)$ for all $t\in\mathbb{R}$.
Taking lifts $\alpha^i,\tau\alpha^j:\mathbb{R}\to \hat{S}$ of these parametrizations with
$\alpha^i(0)=x^i_\alpha$, $\tau\alpha^j(0)=x^j_{\tau\alpha}$ we state that
$\alpha^i(t)$ is related by $\hat{\tau}$ with $\tau\alpha^j(t)$ for all $t\in\mathbb{R}$.
We want to prove that $(\hat{\D},\hat{\tau})$ is an abstract Johansson diagram on $\hat{S}$.
\begin{claim}\label{claim:2-hat-tau-well-defined}
The sistering $\hat{\tau}$ is well defined: if $x_{\alpha}^i,x_{\alpha}^{i'}$ are different
lifts of $x_\alpha$ lying on the same curve $\hat{\alpha}$ of $\hat{\D}$, then
$x_{\tau\alpha}^{j}$ and $x_{\tau\alpha}^{j'}$ lie on the same curve of
$\hat{\D}$, where $j=\rho(a)(i)$ and $j'=\rho(a)(i')$.
\end{claim}
\begin{proof}[Proof of Claim~\ref{claim:2-hat-tau-well-defined}]
Assume first that $x_{\alpha}^i$ and $x_{\alpha}^{i'}$ are consecutive in
$\hat{\alpha}$: we can travel along $\hat{\alpha}$ from
$x_{\alpha}^i$ to $x_{\alpha}^{i'}$ without crossing any other point of
$h^{-1}(x_{\alpha})$. Choosing appropiately the orientation of
$\alpha$ or the indices $i,i'$, we can assume that the lift of $\alpha$ starting
at $x_{\alpha}^i$ ends at $x_{\alpha}^{i'}$, and in terms of $\pi(\D)$ this
implies that $\rho(\alpha)(i)=i'$. Therefore,
\begin{align*}
\rho(\tau\alpha)(j)&=\rho(\tau\alpha)(\rho(a)(i))=\rho(a\tau\alpha)(i)\\
& =\rho(\alpha a)(i)= \rho(a)(\rho(\alpha)(i))=\rho(a)(i')=j'\,,
\end{align*}
because by the double curve relation \eqref{eq:double-curve-relator} it is
$\alpha a=a \tau\alpha$.
The points $x_{\tau\alpha}^j$ and $x_{\tau\alpha}^{j'}$ are consecutive in a lift of $\tau\alpha$ in $\hat{\D}$.

By repeating this argument, the statement is also true if
$x_{\alpha}^i$ and $x_{\alpha}^{i'}$ are not consecutive in
$\hat{\alpha}$
\end{proof}
\begin{claim}\label{claim:1-hat-tau-involutive}
The sistering $\hat{\tau}$ is an involution on $\hat{\D}$.
\end{claim}

\begin{proof}[Proof of Claim~\ref{claim:1-hat-tau-involutive}]
  We have $\tau a=a^{-1}$ in $\pi(\D)$ by dual loop relations. Then
$\rho(\tau a)=\rho(a)^{-1}$ because $\rho$ is a homomorphism.
Therefore if we have $\rho(a)(i)=j$, it is also $\rho(\tau a)(j)=i$. In other words, the
sister of the curve of $\hat{\D}$ passing through $x_\alpha^i$ is the curve of
$\hat{\D}$ passing through $x_{\tau\alpha}^j$, and the sister of the curve of
$\hat{\D}$ passing through $x_{\tau\alpha}^j$ is the curve of $\hat{\D}$ passing
through $x_{\alpha}^i$.
\end{proof}

The sistering $\hat{\tau}$ defines an equivalence relation into the points of
$\hat{\D}$. 
By Claims~\ref{claim:2-hat-tau-well-defined} 
and~\ref{claim:1-hat-tau-involutive}, each point of $\hat{\D}$ which is not a 
crossing has exactly another point related with it.

\begin{claim}
  \label{claim:4-crossings-in-triples}
  The crossings of $\hat{\D}$ are related in triplets by $\hat{\tau}$.
\end{claim}

\begin{proof}[Proof of Claim~\ref{claim:4-crossings-in-triples}]
Let $P$ be a triple point of $\Sigma$, and let $P_1,P_2,P_3$ be the
three crossings of $\D$ in the triplet of $P$. We label the curves of
$\D$ intersecting at $P_1,P_2,P_3$ as in Figure~\ref{Fig:triple-point-diagram}, and we consider the paths $$\omega_1,\omega_2,\omega_3,\quad d_{\alpha},d_{\tau\alpha},d_{\beta},d_{\tau\beta},d_{\gamma},d_{\tau\gamma},\quad
t_{\alpha},t_{\tau\alpha},t_{\beta},t_{\tau\beta},t_{\gamma},t_{\tau\gamma}$$
as in the construction of the triple point relations (Figure~\ref{Fig:triple-point-diagram}).

For each $n=1,2,3$ and $i=1,2,\ldots$, let $\omega_n^i$ be the lift of $\omega_n$ starting at $x^{i}$, and let $P_n^i$ be the endpoint of $\omega_n^i$.
If we choose one of the lifts $P_1^i$ of $P_1$, it is related by $\hat{\tau}$ with one lift $P_2^j$ of $P_2$ by means of the sister curve of the lift of $\alpha$ passing through $P_1^i$.
Let label $\alpha^i$ the lift of $\alpha$ passing through $P_1^i$.
In order to find $P_2^j$ we need to find the sister curve of $\alpha^i$,
and to do this we follow $\alpha^i$ until we reach a lift of $x_\alpha$.
This is made by taking the lift of $d_\alpha$ starting at $P_1^i$,
which ends at a lift $x_\alpha^{i'}$ of $x_\alpha$.
The lift of $\lambda_\alpha$ ending at $x_\alpha^{i'}$ starts at $x^{i'}$,
where $i'=\rho(t_\alpha)(i)$.
The sister curve of $\alpha^i$ is the lift of $\tau \alpha$ 
passing through $x_{\tau\alpha}^{j'}$, where $j'=\rho(a)(i')$.
Finally, $P_2^j$ is located at the starting point of the lift of 
$d_{\tau \alpha}$ ending at $x_{\tau\alpha}^{j'}$, 
and it verifies that $j=\rho(t_{\tau\alpha}^{-1})(j')$. 
In other words, it is
$$j=\rho(t_\alpha a t_{\tau\alpha}^{-1})(i)\,.$$
By the same argument, we have that $P_2^j$ is related with $P_3^\ell$, with $\ell=\rho(t_\beta b t_{\tau\beta}^{-1})(j)$ and that $P_3^\ell$ is related with $P_1^{i^*}$, where $i^*=\rho(t_\gamma c t_{\tau\gamma}^{-1})(\ell)$.
Triple point relation~\eqref{ec:3pto_relation} at $P$ implies that $i^*=i$, and thus there is no more points related with $P_1^i$ different from $P_2^j$ and $P_3^\ell$. This construction is valid in general and so the crossings of $\D$ are related in triples by $\hat{\tau}$.
\end{proof}

We have proved that $(\hat{\D},\hat{\tau})$ is an abstract Johansson diagram in 
$\hat{S}$: points in $\hat{\D}$ are identified in pairs with the exception of 
crossings, which are identified in triplets. Moreover, this equivalence 
relation is compatible with $f$, that is, if $\hat{Y},\hat{Z}\in\hat{S}$ are 
related by $\hat{\tau}$, their projections $h(\hat{Y}),h(\hat{Z})\in S$ are 
related by $\tau$. The quotient space $\hat{\Sigma}$ of $\hat{S}$ under the 
equivalence relation defined by $\hat{\tau}$ is a (pseudo) Dehn surface and we 
can define a map $\hat{h}:\hat{\Sigma}\to\Sigma$ making the following diagram commutative:
\[
  \xymatrix{
    \hat{S} \ar[r]^{\hat{f}} \ar[d]_{h} & \hat\Sigma \ar[d]^{\hat{h}} \\
    S       \ar[r]_ {f}      & \Sigma
  }
\]
where $\hat{f}:\hat{S}\to\hat\Sigma$ is the quotient map. The map $\hat{h}$ is 
in fact an $m$-sheeted covering map whose monodromy homomorphism $\hat\rho$ 
verifies by construction $\rho=\hat\rho\circ\varepsilon$.
\end{proof}

\begin{proof}[Proof of Theorem~\ref{thm:fundamental-group-presentation}]
  Due to Cayley's Theorem ~\cite[p. 9]{Hall}, every group admits an injective 
  representation into the group of permutations of its own elements. Because 
  $\pi(\D)$ is finite or countable, the group of permutations of its elements 
  is isomorphic to $\Omega_m$, for an $m\in\NI$. Thus, we can construct a 
  faithful representation $\rho:\pi(\D)\to\Omega_m$, and by 
  Lemma~\ref{lem:representations-factorize}, this implies that $\varepsilon$ 
  must be injective.
\end{proof}

Hurewicz's Theorem implies that the first homology group of 
$\Sigma$ is isomorphic to the abelianization of its fundamental group. Let us 
see how to adapt the presentation of 
Theorem~\ref{thm:fundamental-group-presentation} to the abelian case. Following 
the notation of above:
\begin{enumerate}[\em({AR}1)]
  
  \item \emph{Dual loop relations}.\label{AR1}
   These relations remains the same, if 
    $\alpha$ and $\tau\alpha$ are sister curves, their preferred dual loops 
    satisfy $a = -\tau a$.
  
  \item \emph{Triple point relations}. The triple points relations can be 
    simplified using commutativity to 
    \[
      a + b + c = w,
    \]
    where $w$ is a word in surfacewise generators and $a$, $b$ and $c$ are 
    preferred dual loops of double curves.
  
  \item \emph{Double curve relations}. Conjugacies on commuting elements are 
    equalities, so the double curve relation can be written as 
    $\alpha = \tau\alpha$.
  
  \item \emph{Surface relations}. The generators of the fundamental group of 
    the $g$-torus $S$ are related by a product of commutators. In the abelian 
    context, this relation is trivial.
  
\end{enumerate}
Relations~AR\ref{AR1} allow us to select $k$ generators $a_1,a_2,\ldots,a_k$ from the set of preferred dual loops, dropping $k$ of them. Finally,

\begin{theorem}
  \label{thm:homology}
  The abelian group $H_1(\Sigma)$ is isomorphic to
  \[
    H_1(\D) = \langle\, s_1,s_2,\ldots,s_{2g},a_1,a_2,\ldots,a_k 
                    \mid \text{AR2 and AR3 relations} \,\rangle.
  \]
\end{theorem}

\begin{remark}\label{rmk:homology-M-Sigma}
  If $\Sigma$ is filling, it is the $2$-skeleton of $M$ as a CW-complex, so 
  $\pi_1(\Sigma)=\pi_1(M)$. Hence $H_1(\Sigma)=H_1(M)$, which is isomorphic to 
  the second homology group due to Poincar\'e duality. In this way, 
  Theorem~\ref{thm:homology} characterizes completely the homology of the 
  closed orientable manifold $M$.
\end{remark}

  \section{Checkerboards and homology spheres}\label{sec:checkers-homology-spheres}
  
  Consider a closed orientable $3$-manifold $M$ and a Dehn surface 
  $S\to\Sigma\subset M$. We say \emph{$M$ is checkered by $\Sigma$} if 
  we can assign a 
  color from $\{0,1\}$ to each region of $\Sigma$ such that if two of them 
  share a face they have different colors. In the same way, $S$ is checkered by 
  the Johansson diagram $\D$ of $\S$ if we can assign a color from $\{0,1\}$ to 
  each face of $S-\D$ such that if two of them share an edge they have 
  different colors. Checkerboards are a usual tool in knot theory 
  (see \cite{Carter-Kamada-Saito}, for example).
    
  \begin{lemma}
    \label{lem:checkered_paths}
    The manifold $M$ is checkered by $\Sigma$ if and only if any loop 
    on $M$ meeting $\Sigma$ transversely on the faces intersects $\S$ in an even 
    number of points.
  \end{lemma}
  
  \begin{proof}
    Given the decomposition of $M$ in regions of $\S$, we can build a graph $G_\Sigma$ 
    describing it combinatorially. Each vertex corresponds to a 
    region and there is an edge connecting two vertices if the corresponding 
    regions share a face. It is clear that $M$ is checkered by $\Sigma$ if and only if
    $G_\Sigma$ is $2$-colourable.
    
    Each path $\gamma$ in $M$ meeting $\Sigma$ transversely on the faces has an associated
    path $c_\gamma$ in $G_\Sigma$ that encodes the sequence of ``crossings'' of $\gamma$ between
    different regions of $\Sigma$ across the faces of $\Sigma$. The number of points of $\gamma\cap\Sigma$ is just the 
    number of edges of $c_\gamma$, and if $\gamma$ is closed, $c_\gamma$ is closed too.
        
    As a graph is $2$-colorable if and only if it is bipartite, the statement 
    of the lemma is equivalent to: $G_\Sigma$ is bipartite if and only if all 
    loops have an even number of edges, which is trivially true.
  \end{proof}
  
  \begin{proposition}
    \label{prop:checkered}
    A $\Z/2$-homology $3$-sphere is checkered by any 
    Dehn surface.
  \end{proposition}
  
  \begin{proof}

%
    Let $M$ be a $\Z/2$-homology $3$-sphere and let $\S$ be a Dehn surface in 
    $M$. Let $\gamma$ be a loop in $M$ intersecting $\Sigma$ transversely at 
    the faces. Both $\gamma$ and $\Sigma$ are $\Z/2$-nullhomologous. 
    Therefore the set of intersection points of $\gamma$ and $\Sigma$ must be a 
    $\Z/2$-nulhomologous $0$-cycle, and then, the number of intersection points 
    should be even. Lemma~\ref{lem:checkered_paths} completes the proof.
  \end{proof}

  If the surface $S$ is orientable, we can fix a well defined normal field on 
  $\Sigma$ in $M$. We assign to each face of $\Sigma$ the color of 
  the region pointed by the normal vector field on the face. The faces of $S$ 
  sharing an edge cannot share also the color, as the corresponding regions 
  share a face (see Figure~\ref{fig:checkers}). Therefore:
  \begin{figure}
    \centering\includegraphics{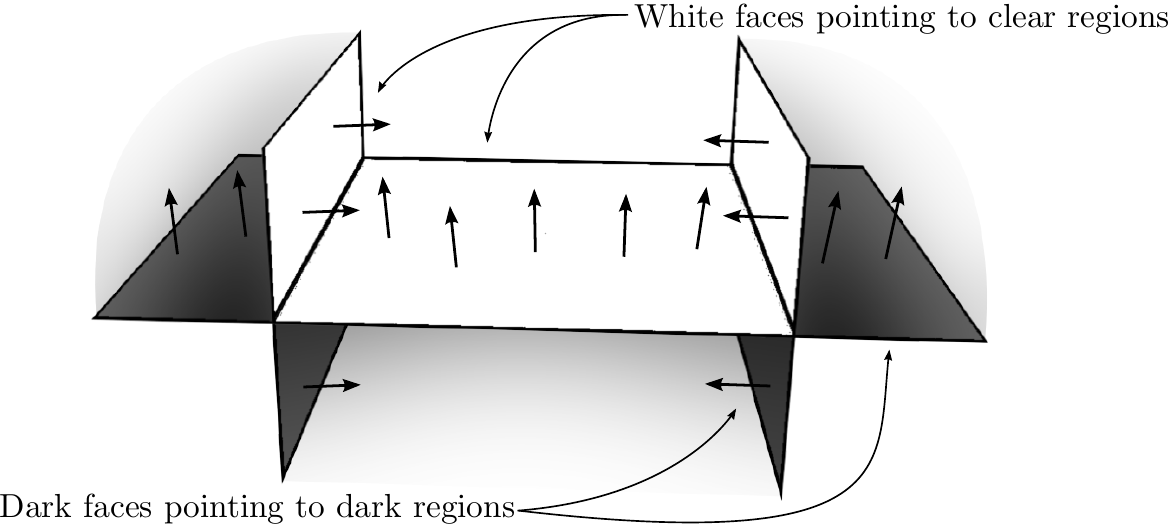}

    \caption{Coloring faces with region colors.}
    \label{fig:checkers}
  \end{figure}
  
  \begin{lemma}\label{lem:orientable-checkered-surface-checkered-diagram}
    If $M$ and $S$ are orientable and $M$ is checkered by $\Sigma$, the 
    surface $S$ is also checkered by $\D$.
  \end{lemma}

  The converse of this lemma does not hold. The 2-sphere $\mathbb{S}^2$ is 
  checkered by the diagram of Figure~\ref{Fig:audi}, but 
  $\mathbb{S}^2\times \mathbb{S}^1$ is not checkered by the corresponding 
  filling Dehn sphere. The existence of checkerboards has strong 
  implications on the number of triple points.
  
  \begin{proposition}
    \label{prop:checkered-even}
    If $M$ is checkered by $\Sigma$, the number of triple points of $\Sigma$ is 
    even.
  \end{proposition}
  
  \begin{proof}
    For each curve of $\D$, consider a \emph{neighbouring curve} running parallel to it and intersecting $\D$ transversely near the crossings of $\D$
    (see~\cite[Fig. 26]{RHomotopies}). As $S$ is 
    checkered by $\D$, by a similar argument to that of the proof of 
    Lemma~\ref{lem:checkered_paths} each neighbouring curve must intersect $\D$ 
    in an even number of points.
    Because the number of curves in $\D$ is even, the total amount of intersection 
    points between neighbouring curves and diagram curves is a multiple of $4$.
    Each crossing of the diagram 
    corresponds to two of those intersections, so the number of crossings is 
    even. Finally, the number of triple points is also even, as it is the 
    number of crossings over three.
  \end{proof}
  
  From Remark~\ref{rmk:at-least-1+2g-triple-points} and Proposition~\ref{prop:checkered-even} we get:
  \begin{theorem}
    \label{thm:spectrum_bound}
    Let $\Sigma$ be a filling Dehn $g$-torus
    of $M$. If $M$ is checkered by $\Sigma$, the number of 
    triple points of $\S$ is greater than or equal to $2g$. \qed
    \end{theorem}
    
    All these results allows us to give a sharper lower bound for the 
    triple point numbers of $\Z/2$-homology $3$-spheres.
    
    \begin{theorem}\label{thm:spectrum-homology-sphere_bound}
    If $M$ is a $\Z/2$-homology $3$-sphere
    \begin{equation*}
      t_g(M) \geq 2 + 2g.
    \end{equation*}
  \end{theorem}
    
   \begin{proof}
    By Theorem~\ref{thm:spectrum_bound} it is enough to show that 
    the number of triple points cannot be $2g$.
    Let $\S$ be a Dehn $g$-torus in $M$ with exactly $2g$ triple 
    points. Let $S$ be the domain of $\S$.
    
    Let us compute the first homology group of $\S$ with coefficients in 
    $\Z/2$. By Theorem~\ref{thm:homology} this abelian group is 
    generated by $s_1,s_2,\ldots,s_{2g}$, the curves generating $H_1(S)$, and the
    preferred dual loops $a_1,a_2,\ldots,a_{k}$ extracted from the $k$ double curves of $\S$. 
    The relations come from the $2g$ AR2 relations and the $k$ AR3 relations by 
    just taking coefficients in $\Z/2$. With this presentation $H_1(\S;\Z/2)$ 
    has the same number of generators as of relators.
        
    Consider only the $k$ AR3 relators. If $\alpha$ and $\tau\alpha$ 
    are sister curves, as they are surfacewise curves they can be written 
    (up to homology, see \eqref{eq:double-curve-relator}) as
    \begin{align*}
      \alpha     &= \sum_{i=1}^{2g} n^i_\alpha s_i,    &
      \tau\alpha &= \sum_{i=1}^{2g} m^i_\alpha s_i,
    \end{align*}
    with $n_\alpha^i,m_\alpha^i\in\Z/2$. Therefore, the relation AR3 relation
    given by the pair $\alpha,\tau\alpha$ is
    \[
      \sum_{i=1}^{2g} q^i_\alpha s_i=0\,,
    \]
    with $q^i_\alpha=n^i_\alpha + m^i_\alpha$.
    All this information can be written as a tableau
    \[
      \begin{array}{c|cccc}
                 & s_1            & s_2            & \cdots & s_{2g} \\[.5mm]
        \hline
        \alpha_1\vphantom{\raisebox{2mm}1}
                 & q_{\alpha_1}^1 & q_{\alpha_1}^2 & \cdots & q_{\alpha_1}^{2g} 
                                                                         \\[1mm]
        \alpha_2 & q_{\alpha_2}^1 & q_{\alpha_2}^2 & \cdots & q_{\alpha_2}^{2g} 
                                                                              \\
        \vdots   & \vdots         & \vdots         & \ddots & \vdots     \\[1mm]
        \alpha_k & q_{\alpha_k}^1 & q_{\alpha_k}^2 & \cdots & q_{\alpha_k}^{2g} 
                                                                         \\[1mm]
      \end{array}
    \]
    Rows in this tableau represents double curves relators.
    Using intersection theory, it can be seen that columns count 
    intersection between the surfacewise generators and diagram curves. 
    By Proposition~\ref{prop:checkered} and Lemma~\ref{lem:orientable-checkered-surface-checkered-diagram} 
    the surface $S$ is checkered by $\D$, so the intersection number of each curve $s_i$ 
    with 
    the diagram should be even. This implies that the sum of each column is $0\pmod2$. 
    In other words, one of the AR3 relators can be written as a linear combination 
    of the rest. Therefore, the group $H_1(\S;\Z/2)$ has 
    $2g+k$ generators and at most $2g+k-1$ nontrivial relations, and it cannot be trivial. 
    By Remark~\ref{rmk:homology-M-Sigma} this implies that $\S$ cannot fill the $\Z/2$-homology $3$-sphere $M$.
  \end{proof}


  It is well known that $t_0(\mathbb{S}^3)=2$ (see~\cite{racsam,tesis}). 
  From Lemma~\ref{lem:triple-point-inequality} and
  Theorem~\ref{thm:spectrum-homology-sphere_bound}, $\mathbb{S}^3$ cannot have exceptional Dehn $g$-tori, therefore $\H(\mathbb{S}^3)=0$ and finally
  
  \begin{corollary}
    $\TT(\mathbb{S}^3)=(2,4,6,\ldots)$.\qed
  \end{corollary}
  
\bibliographystyle{amsplain}

\end{document}